\documentclass[12pt]{amsart}
\usepackage[all]{xy}

\usepackage{amsmath, amsthm, amssymb,latexsym, graphics}

\newcommand{\N}{\mathbb{N}}
\newcommand{\Z}{\mathbb{Z}}

\newcommand{\R}{\mathbb{R}}
\newcommand{\C}{\mathbb{C}}
\newcommand{\E}{\mathbb{E}}

\newcommand{\spr}[2]{\langle #1, #2 \rangle}

\newcommand{\HI}{H^\infty}
\newcommand{\Wa}{W^\alpha}
\newcommand{\Hor}{\mathcal{H}}
\newcommand{\Ha}{\Hor^\alpha}
\newcommand{\Hae}{\Hor^{\alpha + \epsilon}}
\DeclareMathOperator{\M}{M}
\newcommand{\mat}{{\text{mat-}\gamma}}

\DeclareMathOperator{\supp}{supp}
\DeclareMathOperator{\Id}{id}
\DeclareMathOperator{\Rad}{Rad}
\DeclareMathOperator{\Gauss}{Gauss}
\DeclareMathOperator{\vect}{span}
\DeclareMathOperator{\Prob}{Prob}

\let\Re=\relax \DeclareMathOperator{\Re}{Re}

\newcommand{\bignorm}[1]{\bigl\Vert#1\bigr\Vert}
\newcommand{\Bignorm}[1]{\Bigl\Vert#1\Bigr\Vert}

\newtheorem{thmalt}{Theorem}[section]

\theoremstyle{definition}
\newtheorem{rem}[thmalt]{Remark}
\newtheorem{defi}[thmalt]{Definition}
\newtheorem{thm}[thmalt]{Theorem}
\newtheorem{cor}[thmalt]{Corollary}
\newtheorem{lem}[thmalt]{Lemma}
\newtheorem{prop}[thmalt]{Proposition}
\newtheorem{assumption}[thmalt]{Assumption}

\numberwithin{equation}{section}

\title[H\"ormander type functional calculus and square function estimates]% end with percent
 {H\"ormander type functional calculus and square function estimates} % This is the full title of the paper

\author{Ch. Kriegler}
\address{Ch. Kriegler\\
Laboratoire de Math\'ematiques (CNRS UMR 6620)\\
Universit\'e Blaise-Pascal (Clermont-Ferrand 2)\\
Campus des C\'ezeaux\\
63177 Aubi\`ere Cedex\\
France
}
\email{christoph.kriegler@math.univ-bpclermont.fr}

\date{\today}
\subjclass[2010]{47A60, 47A80, 46J15, 42B15}
\keywords{Functional calculus, Square functions, H\"ormander spectral multipliers, Operator spaces}

\allowdisplaybreaks

\begin{document}

\begin{abstract}
We investigate H\"ormander spectral multiplier theorems as they hold on $X = L^p(\Omega),\: 1 < p < \infty,$ for many self-adjoint elliptic differential operators $A$ including the standard Laplacian on $\R^d.$
A strengthened matricial extension is considered, which coincides with a completely bounded map between operator spaces in the case that $X$ is a Hilbert space.
We show that the validity of the matricial H\"ormander theorem can be characterized in terms of square function estimates for imaginary powers $A^{it}$, for resolvents $R(\lambda,A),$ and for the analytic semigroup $\exp(-zA).$
We deduce H\"ormander spectral multiplier theorems for semigroups satisfying generalized Gaussian estimates.
\end{abstract}

\maketitle

\section{Introduction}\label{Sec 1 Intro}

Let $f$ be a bounded function on $(0,\infty)$ and $u(f)$ the operator on $L^p(\R^d)$ defined by $[u(f)g]\hat{\phantom{i}}(\xi) = f(|\xi|^2) \hat{g}(\xi).$
H\"ormander's theorem on Fourier multipliers \cite[Theorem 2.5]{Hor} asserts that $u(f) : L^p(\R^d) \to L^p(\R^d)$ is bounded for any $p \in (1,\infty)$ provided that for some integer $N$ strictly larger than $\frac{d}{2}$
\begin{equation}\label{Equ Hormander original}
\sup_{R > 0} \int_{R/2}^{2R} \left|t^k f^{(k)}(t)\right|^2 \frac{dt}{t} < \infty \quad \left(k = 0,1,\ldots,N\right).
\end{equation}

This theorem has many refinements and generalisations to various similar contexts.
For $\alpha > \frac12,$ let $\Wa_2(\R) = \{ f \in L^2(\R):\: \|f\|_{\Wa_2(\R)} = \| (1 + \xi^2)^{\alpha/2} \hat{f}(\xi)\|_{L^2(\R)} < \infty \}$ denote the usual Sobolev space, and $\Wa = \{  f : (0,\infty) \to \C:\: f \circ \exp \in \Wa_2(\R)\},$ which is a Banach algebra with respect to $\|f\|_{\Wa} = \|f \circ \exp\|_{\Wa_2(\R)}.$
Let $\phi_0 \in C^\infty_c(\frac12,2).$
For $n \in \Z,$ let $\phi_n = \phi_0(2^{-n} \cdot)$ and assume that $\sum_{n \in \Z} \phi_n(t) = 1$ for any $t > 0.$
Such a function exists \cite[Lemma 6.1.7]{BeL} and we call $(\phi_n)_{n \in \Z}$ a dyadic partition of unity.
We define the Banach algebra
\[ \Ha = \left\{ f : (0,\infty) \to \C :\: \| f\|_{\Ha} = \sup_{n \in \Z} \|\phi_n f\|_{\Wa} < \infty \right\}.\]
The definition of $\Ha$ is independent of the dyadic partition of unity, different choices resulting in equivalent norms \cite[Section 4.2]{Kr}.
The space $\Ha$ refines \eqref{Equ Hormander original}, more precisely, $f \in \Ha$ implies that $f$ satisfies \eqref{Equ Hormander original} for $N \leq \alpha,$ and the converse holds for $N \geq \alpha$ \cite[Proposition 4.11]{Kr}.

Now if $A$ is a self-adjoint positive operator on some $L^2(\Omega,\mu),$ then its functional calculus assigns to any bounded measurable function $f$ on $(0,\infty)$ an operator $f(A)$ on $L^2(\Omega,\mu).$
In particular, if $A = - \Delta$ and $(\Omega,\mu) = (\R^d,dx),$ then $f(A)$ equals the above $u(f).$
A  theorem of H\"ormander type holds true for many elliptic differential operators $A,$ including sublaplacians on Lie groups of polynomial growth, Schr\"odinger operators and elliptic operators on Riemannian manifolds \cite{DuOS,Alex,Bluna,Chri,Duon}.
By this, we mean that 
\begin{equation}\label{Equ Intro Hormander calculus}
u : \Ha \to B(X),\, f \mapsto f(A)\text{ is a bounded homomorphism,}
\end{equation}
where $X = L^p(\Omega),\,p \in (1,\infty),$ $\alpha$ is the differentiation parameter typically larger than $\frac{d}{2},$ where $d$ is the dimension of $\Omega,$ and $f(A)$ is
(the unique bounded $L^p$-extension of) the self-adjoint functional calculus.\\

The aim of this article is to characterize the validity of the H\"ormander multiplier theorem for $A$ in terms of square function estimates.

The latter have been introduced in Stein's classical book \cite{Ste1} and have since then been used widely with applications to functional calculi and multiplier theorems.
Note that $\|(\cdot)^{it}\|_{\Ha} \cong ( 1 + |t|^2)^{\alpha/2}$ \cite[Proposition 4.12 (4)]{Kr}, so that for this particular function, \eqref{Equ Intro Hormander calculus} implies $\|A^{it}\| \leq C ( 1 + |t|^2 )^{\alpha/2}.$
Then a natural square function estimate for our situation is
\begin{equation}\label{Equ Intro Square Function}
\| (1 + t^2)^{-\alpha/2} A^{it} x \|_{\gamma(\R,X)} \leq C \| x\|_X ,
\end{equation}
where $\gamma(\R,X)$ is given by
\[ \| x(t) \|_{\gamma(\R,X)} \cong \left\| \left( \int_\R \left|x(t)\right|^2 dt \right)^{\frac12} \right\|_X \]
for $X = L^p(\Omega,\mu)$ and $p \in [1,\infty),$ which explains the name square function.
The general definition of the space $\gamma(\R,X)$ involves Gaussian random sums in the Banach space $X,$ see Section \ref{Sec 2 Prelims}.

Our setting, developed in Section \ref{Sec 2 Prelims}, is as follows:
We let $X$ be a Banach space having Pisier's property $(\alpha),$ which a geometric property playing an important role for the theory of spectral multipliers.
It is natural to assume the operator $A$ to be $0$-sectorial
i.e. a negative generator of an analytic semigroup $(\exp(-zA))_{\Re z > 0}$ which is uniformly bounded on the sector $\Sigma_\omega = \{ z \in \C \backslash \{ 0 \} :\: | \arg z | < \omega \}$ for each $\omega < \frac{\pi}{2}.$
Indeed, $\exp(-z \cdot)$ belongs to $\Ha$ with uniform norm bound on such sectors.
Further, for simplicity we assume throughout that $A$ has dense range.

We shall base the definition of $u$ in \eqref{Equ Intro Hormander calculus} on the well-known $\HI$ functional calculus \cite{CDMY,KuWe}.
This means that for $f$ belonging to $\HI_0(\Sigma_\omega) = \{ f \in \HI(\Sigma_\omega) :\: \exists\: \epsilon,\, C > 0 \text{ s.th. }|f(z)| \leq C \min ( |z|^{\epsilon}, |z|^{-\epsilon} ) \}$ which is a subclass of
$\HI(\Sigma_\omega) = \{ f : \Sigma_\omega \to \C :\: f\text{ is analytic, }\|f\|_{\infty,\omega} = \sup_{z \in \Sigma_\omega} | f(z) | < \infty \},$ $f(A) \in B(X)$ is defined by a certain Cauchy integral formula, see \eqref{Equ Cauchy integral formula}.
Secondly, under certain conditions, $A$ has a bounded $\HI$ calculus, which means that
there is an extension to a bounded homomorphism $\HI(\Sigma_\omega) \to B(X),\,f \mapsto f(A).$
Note that $\HI(\Sigma_\omega)$ is a subclass of $\Ha.$
In Lemma \ref{Lem A u} it will be shown in particular that an extension of the $\HI$ calculus to a bounded homomorphism $u : \Ha \to B(X)$ is unique.

For any such mapping $u$ and $n\in \N,$ we now consider the linear tensor extension
\[ u_n : \begin{cases} M_n  \otimes \Ha & \to M_n \otimes B(X) \\ a \otimes f & \mapsto a \otimes u(f) \end{cases},\]
where $M_n$ is the space of $n \times n$ scalar matrices.
We will equip both $M_n \otimes \Ha$ and $M_n \otimes B(X)$ with suitable norms.
In fact, $\Ha$ will become an operator space (see Section 4), $M_n \otimes B(X) \cong B(\ell^2_n \otimes_2 X)$ if $X$ is a Hilbert space, 
and if $X$ is a Banach space, $M_n \otimes B(X) \cong B(\Gauss_n(X))$ carries the norm induced by an action on $X$-valued Gaussian random sums.
We call $u$ matricially $\gamma$-bounded in this article if
\begin{equation}\label{Equ Intro mat-gamma}
\|u\|_\mat = \sup_{n \in \N} \| u_n \| < \infty.
\end{equation}
This is in general strictly stronger than $\|u\| < \infty$ (see Proposition \ref{Prop bounded vs mat-gamma bounded}), and is related to the following two well-known boundedness notions, explained in Section \ref{Sec 2 Prelims}.
First, if $X$ is a Hilbert space, then \eqref{Equ Intro mat-gamma} is equivalent to the complete boundedness of $u,$ and 
second, if $X$ is a Banach space, then \eqref{Equ Intro mat-gamma} entails that the set of spectral multipliers $\{ u(f) : \: \|f\|_{\Ha} \leq 1\}$ is $\gamma$-bounded.

The main result reads as follows.

\begin{thm}\label{Thm Intro Main}
Let $X$ be a space with property $(\alpha).$
Let $A$ be a $0$-sectorial operator on $X$ with bounded $\HI$ calculus.
Let $\alpha > \frac12.$
Then the following are equivalent.
\begin{enumerate}
\item The square function estimate \eqref{Equ Intro Square Function} holds.
\item The $\HI$ calculus mapping $f\mapsto f(A)$ extends to a homomorphism $u : \Ha \to B(X)$ which is matricially $\gamma$-bounded.
\end{enumerate}
\end{thm}

Theorem \ref{Thm Intro Main} entails a spectral multiplier theorem in the following situations:
The space $X = L^p(\Omega)$ for $p \in (1,\infty)$ has property $(\alpha).$
If $(\Omega,\mu)$ is a $d$-dimensional space of homogeneous type, e.g. a sufficiently regular open subset of $\R^d$ with Lebesgue measure $\mu,$ and $A$ is self-adjoint positive on $L^2(\Omega)$ such that the corresponding semigroup $\exp(-tA)$ has an integral kernel $k_t(x,y)$ that satisfies the Gaussian estimate for some $m \in \N$
\begin{equation}\label{Equ Intro GE}
|k_t(x,y)| \leq C \mu(B(x,t^{\frac1m}))^{-1} \exp\left( -c (\text{dist}(x,y)/t^{\frac1m})^{\frac{m}{m-1}} \right) \quad (x,y \in \Omega,\: t > 0),
\end{equation}
then $A$ has a bounded $\HI$ calculus on $X$
\cite[Theorem 3.4]{DR}, \cite[Corollary 2.3]{Blun}.
This is indeed the case for many operators listed before \eqref{Equ Intro Hormander calculus}
\cite[Section 2]{Bluna}.
Moreover, the mappings $u$ from \eqref{Equ Intro Hormander calculus} and Theorem \ref{Thm Intro Main} (2) are the same, so that we obtain as a corollary

\begin{cor}\label{Cor Intro}
Assume that $A$ is a self-adjoint positive operator on $L^2(\Omega)$ satisfying \eqref{Equ Intro GE}.
Let $\alpha > \frac12$ and $p \in (1,\infty).$
If $A$ satisfies the square function estimate
\begin{equation}\label{Equ Intro Cor}
\left\| \left( \int_\R \left|(1 + t^2 )^{-\alpha/2} A^{it} x \right|^2 dt \right)^{\frac12} \right\|_p \leq C \|x\|_p, 
\end{equation}
then for any $f \in \Ha,$ the spectral multiplier $f(A)$ is bounded $ L^p(\Omega) \to L^p(\Omega).$
\end{cor}
In Proposition \ref{Prop bounded vs mat-gamma bounded}, we will show a partial converse of Corollary \ref{Cor Intro}.
More precisely, \eqref{Equ Intro Hormander calculus} implies that a restriction to a smaller H\"ormander space $\Hor^\beta$ is matricially $\gamma$-bounded.

Let us close the introduction with an overview of the rest of the article.
In Section \ref{Sec 2 Prelims}, we give the necessary background of the above mentioned notions of matricial norms, square functions, Gaussian random sums and functional calculus.
Matricially $\gamma$-bounded mappings and the connection to square functions are explained in Section \ref{Sec 3 Main Thm}.
Section \ref{Sec 4 Hormander operator space} is devoted to homomorphisms $u : \Ha \to B(X)$ and the connection to $\HI$ functional calculus.
Moreover Theorem \ref{Thm Intro Main} is proved.
A main ingredient is to deduce a spectral decomposition of Paley-Littlewood type, see \eqref{Equ Paley Littlewood}, under the hypotheses of Theorem \ref{Thm Intro Main}.
In Section \ref{Sec 5 Examples}, we discuss some extensions and applications.
Firstly, the square function estimate in terms of imaginary powers $A^{it}$ in Theorem \ref{Thm Intro Main} has several equivalent and almost equivalent rewritings in terms of other typical square functions, involving the analytic semigroup
\begin{align}
\left\| A^{\frac12} \exp(-te^{i\theta} A)x \right\|_{\gamma(\R_+,X)} & \leq C (\frac{\pi}{2} - |\theta| )^{-\beta} \|x\| & \quad \left(\theta \in (-{\pi}/{2},{\pi}/{2}) \right),
\label{Equ Intro Sgr Square Function} \\
\intertext{or resolvents}
\|A^{\frac12} R(e^{i\theta}t,A) x \|_{\gamma(\R_+,X)}  & \leq C |\theta|^{-\beta} \|x\| & \quad \left(\theta \in (-\pi,\pi) \backslash \{ 0 \} \right). \label{Equ Intro Res Square Function}
\end{align}
We have \eqref{Equ Intro Square Function} $\Rightarrow$ \eqref{Equ Intro Sgr Square Function} and \eqref{Equ Intro Res Square Function} for $\alpha \leq \beta,$ and conversely, \eqref{Equ Intro Sgr Square Function} or \eqref{Equ Intro Res Square Function} $\Rightarrow$ \eqref{Equ Intro Square Function} for $\alpha > \beta.$

Secondly, we discuss Theorem \ref{Thm Intro Main} in the presence of
generalized Gaussian estimates (see Assumption \ref{Ass Examples}), which in particular covers
semigroups satisfying \eqref{Equ Intro GE}.
This is a well-studied property in connection with (H\"ormander) functional calculus, see e.g. \cite{Duon,DuOS}.
In particular, we show the square function assumption of Corollary \ref{Cor Intro} in the form of \eqref{Equ Intro Sgr Square Function} and improve the derivation order of the H\"ormander theorem from $\alpha > \frac{d}{2} + \frac12$ as proved in \cite{Bluna} to $\alpha > d \left| \frac{1}{p_0} - \frac12 \right| + \frac12.$
We finally discuss the connections and differences between matricially $\gamma$-bounded H\"ormander calculus and bounded H\"ormander calculus.
The last Section \ref{Sec 6 Proofs Lemmas} contains some technical proofs of Section \ref{Sec 4 Hormander operator space}.

\section{Preliminaries on Operator spaces, Gaussian sums, Square functions and Functional calculus}\label{Sec 2 Prelims}

We will need in different contexts cross norms on a tensor product of two Banach spaces.

\subsection*{Operator spaces}

A Banach space $E$ is called operator space if it is isometrically embedded into $B(H),$ where $H$ is a Hilbert space.
Let $M_n$ denote the space of scalar $n \times n$ matrices.
What makes operator spaces different from mere Banach spaces is that there is a specific collection of norms on $M_n \otimes E$, the operator space structure of $E.$
Namely for all $n \in \N,$ it is equipped with the norm
arising from the embedding $M_n \otimes E \hookrightarrow B(\ell^2_n(H)),\, [a_{ij}] \otimes x \mapsto \left( (h_i)_{i=1}^n \mapsto (\sum_{j=1}^n a_{ij}x(h_j))_{i=1}^n \right).$

Let $E$ and $F$ be operator spaces and $u : E \to F$ a linear mapping.
For any $n \in \N,$ let $u_n$ be the linear mapping $M_n \otimes E \to M_n \otimes F,\, a \otimes x \mapsto a \otimes u(x).$
Then $u$ is called completely bounded (completely isometric) if $\|u\|_{cb} = \sup_{n \in \N} \|u_n\| < \infty$ (for any $n \in \N,\:u_n$ is isometric).

Clearly, any space $B(H)$ itself is an operator space, so in particular $M_m = B(\ell^2_m)$ is.
Further we will consider the Hilbert row space $\ell^2_r = \{h \mapsto \spr{h}{x}e :\:x \in \ell^2\} \subset B(\ell^2)$
where $e \in \ell^2$ is a fixed element of norm $1$ and $\spr{h}{x}$ is the scalar product.
Different choices of $e$ give isometric norms of $M_n \otimes \ell^2_r$ and $\ell^2_r$ is isometric to $\ell^2$ as a Banach space.
We shall also consider the $m$-dimensional subspaces $\ell^2_{m,r} \subset \ell^2_r.$
These are completely isometrically determined by the following embedding, which also explains the name of row space:

\begin{equation}\label{Equ i_m}
i_m : \ell^2_m \hookrightarrow M_m,\,(a_1,\ldots,a_m) \mapsto
\left(
  \begin{array}{ccc}
    a_1 & \ldots & a_m \\
    0 & \ldots & 0 \\
    \vdots & \ddots & \vdots \\
    0 & \ldots & 0 \\
  \end{array}
\right).
\end{equation}

We refer to the books \cite{ER,Pis2} for further information on operator spaces.

\subsection*{$\gamma$-bounded sets, property $(\alpha)$ and square functions}

We let $\Omega$ be a probability space and $(\gamma_k)_{k \in \Z}$ a sequence of independent standard Gaussian random variables on $\Omega.$
For a Banach space $X,$ we let $\Gauss(X)$ be the closure of $\vect\{ \gamma_k \otimes x_k :\: k \in \Z\}$ in $L^2(\Omega;X)$ with respect to the norm
\begin{equation}\label{Equ Def Gauss}
\bignorm{ \sum_k \gamma_k \otimes x_k }_{\Gauss(X)}  = \left( \int_\Omega \left\| \sum_{k=1}^n \gamma_k(\omega) x_k \right\|^2 d\omega \right)^{\frac12}.
\end{equation}
It will be convenient to denote $\Gauss_n(X)$ the subspace of $\Gauss(X)$ of elements of the form $\sum_{k=1}^n \gamma_k \otimes x_k.$

Note that if $X$ is a Hilbert space, then
\begin{equation}\label{Equ Hilbert Gaussian}
 \bignorm{ \sum_k \gamma_k \otimes x_k }_{\Gauss_n(X)}^2 = \sum_{k=1}^n \|x_k\|^2.
\end{equation}

A collection $\tau \subset B(X)$ is called $\gamma$-bounded if there exists $C > 0$ such that
\[ \left\| \sum_k \gamma_k \otimes T_k x_k \right\|_{\Gauss(X)} \leq C \left\| \sum_k \gamma_k \otimes x_k \right\|_{\Gauss(X)} \]

for any finite families $T_1,\ldots,T_n \in \tau$ and $x_1,\ldots, x_n \in X.$
The least admissible constant is denoted by $\gamma(\tau)$ (and $\gamma(\tau) := \infty$ if such a $C$ does not exist).
Note that a $\gamma$-bounded set is automatically uniformly norm bounded, since one has $\gamma(\tau) \geq \sup_{T \in \tau} \|T\|.$
For $\sigma,\tau \subset B(X)$ and $\sigma \circ \tau = \{S \circ T :\: S \in \sigma,\,T \in \tau\},$ one has $\gamma(\sigma\circ \tau) \leq \gamma(\sigma)\gamma(\tau).$
The set $\tau = \{ a \Id_X :\: a \in \C,\, |a| \leq 1\}$ is $\gamma$-bounded with constant 1.

We say that $X$ has property $(\alpha)$ if there is a constant $C \geq 1$ such that for any finite family $(x_{ij})$ in $X$, we have
\begin{equation}\label{Equ Property Gauss alpha}
\frac1C \bignorm{ \sum_{i,j} \gamma_{ij} \otimes x_{ij} }_{\Gauss(X)} \leq \bignorm{\sum_{i,j} \gamma_i \otimes \gamma_j \otimes x_{ij} }_{\Gauss(\Gauss(X))} \leq C \bignorm{ \sum_{i,j} \gamma_{ij} \otimes x_{ij} }_{\Gauss(X)},
\end{equation}
where $\gamma_{ij}$ is a doubly indexed family of independent standard Gaussian variables.
Property $(\alpha)$ is inherited by closed subspaces and isomorphic spaces.
The $L^p$ spaces have property $(\alpha)$ for $1 \leq p < \infty$ and moreover, if $X$ has property $(\alpha),$ then also $L^p(\Omega;X)$ has.
Property $(\alpha)$ is usually defined in terms of independent Rademacher variables $\epsilon_i,$ i.e. $\Prob(\epsilon_i = \pm 1) = \frac12$ instead of Gaussian variables \cite{Pis}.
In analogy with \eqref{Equ Def Gauss}, we define $\Rad(X) \subset L^2(\Omega;X)$ by
\[\bignorm{ \sum_k \epsilon_k \otimes x_k }_{\Rad(X)}  = \left( \int_\Omega \left\| \sum_{k=1}^n \epsilon_k(\omega) x_k \right\|^2 d\omega \right)^{\frac12}.\]
It turns out that the two definitions are the same:

\begin{lem}\label{Rem Property Gauss alpha}
The property \eqref{Equ Property Gauss alpha} is equivalent to the following equivalence uniform in finite families $(x_{ij})$ in $X.$
\begin{equation}\label{Equ Property Rad alpha}
 \bignorm{ \sum_{i,j} \epsilon_i \otimes \epsilon_j \otimes x_{ij} }_{\Rad(\Rad(X))} \cong \bignorm{ \sum_{i,j} \epsilon_{ij} x_{ij} }_{\Rad(X)}.
\end{equation}
\end{lem}

\begin{proof}
First observe that the Schatten classes $S^p$ for $p \in (1,\infty) \backslash \{ 2 \}$ are spaces which do not satisfy \eqref{Equ Property Gauss alpha} nor \eqref{Equ Property Rad alpha}.
This is shown in \cite{Pis} for the Rademachers.
On the other hand, $S^p$ has finite cotype, which implies that on this space, Rademacher sums and Gaussian sums are equivalent \cite[Theorem 12.27]{DiJT}, i.e.
\begin{equation}\label{Equ Gauss Rad}
\|\sum_{k \in F} \gamma_k \otimes x_k\|_{\Gauss(X)} \cong \|\sum_{k \in F} \epsilon_k \otimes x_k\|_{\Rad(X)},
\end{equation}
uniformly in $F \subset \Z.$
From this one easily deduces that \eqref{Equ Property Gauss alpha} does not hold.

Next observe that by the Banach-Mazur theorem and \cite[Theorem 3.2]{DiJT}, $S^p$ (in fact any Banach space) has the property that all finite dimensional subspaces are isomorphic to a subspace of some $\ell^\infty_n,$
with one fixed isomorphism constant.
This implies that $\ell^\infty$ does not satisfy \eqref{Equ Property Gauss alpha} nor \eqref{Equ Property Rad alpha}.

Therefore, by the characterization of finite cotype in \cite[Theorem 14.1]{DiJT}, a space $X$ satisfying \eqref{Equ Property Gauss alpha} or \eqref{Equ Property Rad alpha} has finite cotype.
As cited above, Rademacher and Gaussian sums are then equivalent, so the corresponding expressions in \eqref{Equ Property Gauss alpha} and \eqref{Equ Property Rad alpha} are, which shows the lemma.
\end{proof}

We recall the construction of Gaussian function spaces from \cite{KaW2}, see also \cite[Section 1.3]{KNVW}.

Let $H$ be a separable Hilbert space.
We consider the tensor product $H \otimes X$ as a subspace of $B(H,X)$ in the usual way, i.e. by identifying
$\sum_{k=1}^n h_k \otimes x_k \in H \otimes X$ with the mapping $u : h \mapsto \sum_{k=1}^n \spr{h}{h_k} x_k$ for any finite
families $h_1,\ldots,h_n \in H$ and $x_1,\ldots,x_n \in X.$
Choose such families with corresponding $u$, where the $h_k$ shall be orthonormal.
Let $\gamma_1,\ldots,\gamma_n$ be independent standard Gaussian random variables over some probability space.
We equip $H \otimes X$ with the norm
\[ \bignorm{ u }_{\gamma(H,X)} = \bignorm{ \sum_k \gamma_k \otimes x_k }_{\Gauss(X)}.\]
By \cite[Corollary 12.17]{DiJT}, this expression is independent of the choice of the $h_k$ representing $u.$
We let $\gamma(H,X)$ be the completion of $H \otimes X$ in $B(H,X)$ with respect to that norm.
Then for $u \in \gamma(H,X),$ $\|u\|_{\gamma(H,X)} = \bignorm{\sum_k \gamma_k \otimes u(e_k)}_{\Gauss(X)},$ where the $e_k$ form an orthonormal basis of $H$ \cite[Remark 4.2]{KaW2}.\\

A particular subclass of $\gamma(H,X)$ will be important, which is obtained by the following procedure.
Assume that $(\Omega,\mu)$ is a $\sigma$-finite measure space and $H = L^2(\Omega).$
Denote $P_2(\Omega,X)$ the space of Bochner-measurable functions $f: \Omega \to X$ such that $x' \circ f \in L^2(\Omega)$ for all $x' \in X'.$
We identify $P_2(\Omega,X)$ with a subspace of $B(L^2(\Omega),X'')$ by assigning to $f$ the operator $u_f$ defined by
\begin{equation}\label{Equ u_f}
\spr{u_f h}{x'} = \int_\Omega \spr{f(t)}{x'} h(t) d\mu(t).
\end{equation}
An application of the uniform boundedness principle shows that, in fact, $u_f$ belongs to $B(L^2(\Omega),X)$ \cite[Section 4]{KaW2}, \cite[Section 5.5]{Frohl}.
Then we let
\[\gamma(\Omega,X) = \left\{ f \in P_2(\Omega,X):\: u_f \in \gamma(L^2(\Omega),X)\right\}\]
and set
\[\|f\|_{\gamma(\Omega,X)} = \|u_f\|_{\gamma(L^2(\Omega),X)}.\]
The space $\{u_f:\:f \in \gamma(\Omega,X)\}$ is a dense and in general proper subspace of $\gamma(L^2(\Omega),X).$
Resuming the above, we have the following embeddings of spaces, cf. also \cite[Section 3]{LM10}.
\[ L^2(\Omega)\otimes X \to \gamma(\Omega,X) \to \gamma(L^2(\Omega),X) \to B(L^2(\Omega),X).\]

In some cases, $\gamma(L^2(\Omega),X)$ and $\gamma(\Omega,X)$ can be identified with more classical spaces.
If $X$ is a Banach function space with finite cotype, e.g. an $L^p$ space for some $p \in [1,\infty),$ then for any
step function $f = \sum_{k=1}^n x_k \chi_{A_k}: \Omega \to X,$ where $x_k \in X$ and the $A_k$ are measurable and disjoint with $\mu(A_k) \in (0,\infty),$ we have
 (cf. \cite[Remark 3.6, Example 4.6]{KaW2})

\[\|f\|_{\gamma(\Omega,X)}  = \Bignorm{ \sum_k \gamma_k \otimes \mu(A_k)^{\frac12} x_k }_{\Gauss(X)} \!\!\!\!\!\!\!\!
\cong \Bignorm{ \left(\sum_k \mu(A_k) |x_k|^2 \right)^{\frac12}}_X \label{Equ classical square function}
= \Bignorm{ \left( \int_\Omega |f(t)(\cdot)|^2 d\mu(t) \right)^{\frac12} }_X.\]

The second equivalence follows from \cite[Theorem 16.18]{DiJT}.
The last expression above is a classical square function (see e.g. \cite[Section 6]{CDMY}),
whence for an arbitrary space $X,\,\|u\|_{\gamma(H,X)}$ is called (generalized) square function \cite[Section 4]{KaW2}.
In particular, if $X$ is a Hilbert space, then $\gamma(\Omega,X) = L^2(\Omega,X)$ with equal norms.

We have the following well-known properties of square functions.
\begin{lem}\label{Lem Folklore square functions}
Let $(\Omega,\mu)$ be a $\sigma$-finite measure space and $X$ a Banach space with property $(\alpha).$
\begin{enumerate}
\item Suppose that $f_n,f \in P_2(\Omega,X)$ and $f_n(t) \to f(t)$ for almost all $t \in \Omega.$
Then $\|f\|_{\gamma(\Omega,X)} \leq \liminf_n \|f_n\|_{\gamma(\Omega,X)}.$
\item Let $K \in B(H_2,H_1),$ where $H_1,H_2$ are separable Hilbert spaces.
Then for $u \in \gamma(H_1,X)$ we have $u \circ K \in \gamma(H_2,X)$ and $\|u \circ K\|_{\gamma(H_2,X)} \leq \|u\|_{\gamma(H_1,X)} \|K\|.$
\item If $\Omega \to B(X),\: t \mapsto N(t)$ is a strongly continuous map such that $\tau = \{ N(t) :\: t \in \Omega \}$ is $\gamma$-bounded,
and $f \in \gamma(\Omega,X),$ then $\|N \cdot f \|_{\gamma(\Omega,X)} \leq \gamma(\tau) \|f\|_{\gamma(\Omega,X)}.$
\end{enumerate}
\end{lem}

\begin{proof}
As $X$ has property $(\alpha),$ it does not contain $c_0$ isomorphically.
Using this fact, a proof of (1) can be found in \cite[Lemma 4.10]{KaW2}, or in \cite[Proposition 3.18]{vN}.
For (2), we refer to \cite[Proposition 4.3]{KaW2} or \cite[Corollary 6.3]{vN}.
Finally, (3) is proved in \cite[Proposition 4.11]{KaW2}, see also \cite[Theorem 5.2]{vN}.
\end{proof}

\subsection*{Sectorial operators and $\HI$ functional calculus}

Let $\theta \in (0,\pi)$ and $A :\: D(A) \subset X \to X$ a densely defined linear mapping on some Banach space $X.$
$A$ is called $\theta$-sectorial, if
\begin{enumerate}
\item The spectrum $\sigma(A)$ is contained in $\overline{\Sigma_\theta}.$
\item For all $\omega > \theta$ there is a $C_\omega > 0$ such that $\| \lambda (\lambda - A)^{-1} \| \leq C_\omega$ for all $\lambda \in \overline{\Sigma_\omega}^c.$
\item $R(A)$ is dense in $X.$
\end{enumerate}
We call $A$ $0$-sectorial if it is $\theta$-sectorial for all $\theta > 0.$
In the literature, property (3) is sometimes omitted.
It entails that $A$ is injective \cite[Proposition 15.2]{KuWe}.
For such an operator $A$ and $f \in \HI_0(\Sigma_\omega),\: \omega \in (\theta , \pi),$ one defines the operator

\begin{equation}\label{Equ Cauchy integral formula}
f(A) = \frac{1}{2\pi i} \int_{\partial \Sigma_{(\theta + \omega)/2}} f(\lambda) (\lambda - A)^{-1} d\lambda,
\end{equation}

where $\partial\Sigma_{(\theta + \omega)/2}$ is the sector boundary which is parametrized as usual counterclockwise.
It is easy to check that $f(A)$ is bounded and that $u : \HI_0(\Sigma_\omega) \to B(X)$ is a linear and multiplicative mapping.
Suppose that there exists $C > 0$ such that
\begin{equation}\label{Equ Bounded HI calculus}
\| f(A) \| \leq C \|f\|_{\infty,\omega} \quad (f \in \HI_0(\Sigma_\omega)).
\end{equation}
Then there exists an extension of $u$ to a bounded mapping $\HI(\Sigma_\omega) \to B(X),\,f \mapsto f(A),$ satisfying the so-called Convergence Lemma \cite[Lemma 2.1]{CDMY}.

\begin{lem}\label{Lem Convergence Lemma}
Let $(f_n)_{n \in \N}$ be a sequence in $\HI(\Sigma_\omega)$ such that $\sup_{n \in \N} \| f_n \|_{\infty,\omega} < \infty$ and $f_n(\lambda) \to f(\lambda)$ for all $\lambda \in \Sigma_\omega$ and some $f$
(which then necessarily belongs to $\HI(\Sigma_\omega)$).
Then $f(A)x = \lim_{n \to \infty} f_n(A) x$ for any $x \in X.$
\end{lem}

Note that the extension is uniquely determined by Lemma \ref{Lem Convergence Lemma} since for any $f \in \HI(\Sigma_\omega),$
\begin{equation}\label{Equ Convergence lemma}
f_n(\lambda) = f(\lambda) \left( \lambda / (1 + \lambda)^2 \right)^{\frac1n}
\end{equation}
is a sequence in $\HI_0(\Sigma_\omega)$ approximating $f$ in the sense of that lemma.
As a consequence, if \eqref{Equ Bounded HI calculus} is satisfied, then it also holds for any $f \in \HI(\Sigma_\omega).$
In this case, we say that $A$ has a bounded $\HI(\Sigma_\omega)$ calculus, or without precising the angle $\omega \in (\theta,\pi),$ a bounded $\HI$ calculus.

\section{Square function estimate and matricially bounded homomorphism}\label{Sec 3 Main Thm}

Throughout the section, we let $X$ be a Banach space.
For any $n \in \N,$ we identify $M_n \otimes B(X)$ with $B(\Gauss_n(X))$ by associating $[a_{ij}] \otimes T \in M_n \otimes B(X)$ with the operator
\begin{equation}\label{Equ M_n otimes B(X)}
\sum_{k=1}^n \gamma_k \otimes x_k \mapsto \sum_{k,j=1}^n \gamma_k \otimes a_{kj} T(x_j).
\end{equation}
Via this identification, we get a norm on the tensor product space, which we note by $M_n \otimes_\gamma B(X).$

\begin{defi}\label{Def matricially gamma bounded}
Let $E$ be an operator space.
Let further $u : E \to B(X)$ be a linear mapping.
We call $u$ matricially $\gamma$-bounded, if $\Id_{M_n} \otimes u : M_n\otimes E \to M_n \otimes_\gamma B(X)$ is bounded uniformly in $n \in \N,$
i.e. if there is a constant $C > 0$ such that for any $n \in \N,$
\[ \bignorm{\sum_{i,j = 1}^n \gamma_i \otimes u(f_{ij}) x_j }_{\Gauss_n(X)} \leq C \bignorm{ \left[f_{i,j}\right]}_{M_n \otimes E} \bignorm{\sum_{i=1}^n \gamma_i \otimes x_i}_{\Gauss_n(X)} . \]
We denote the least admissible constant $C$ by $\|u\|_{\mat}.$
\end{defi}

\begin{rem}\label{Rem mat}~
\begin{enumerate}
\item If $u : E \to B(X)$ is matricially $\gamma$-bounded then
\begin{equation}\label{Equ u gamma bounded}
 \{ u(x) : \: \|x\|_E \leq 1\}
\end{equation}
is $\gamma$-bounded.
Indeed, from the definition of $\gamma$-boundedness in Section \ref{Sec 2 Prelims}, we immediately deduce that
\eqref{Equ u gamma bounded} is satisfied if and only if $\Id_{M_n} \otimes u|_{D_n \otimes E}$ is bounded uniformly in $n \in \N,$
where $D_n \subset M_n$ denotes the subspace of diagonal matrices.
We call a linear mapping $u$ $\gamma$-bounded if \eqref{Equ u gamma bounded} holds.
\item Assume that $X$ is a Hilbert space.
Then by \eqref{Equ Hilbert Gaussian}, $u$ is matricially $\gamma$-bounded if and only if $u$ is completely bounded, and in this case, $\|u\|_{cb} = \|u\|_\mat.$
\end{enumerate}
\end{rem}

A first example for Definition \ref{Def matricially gamma bounded} is given by
\begin{prop}\label{Prop sigma(m,X)}
For a given space $X$ and $m \in \N,$ consider
\[
\sigma_{m,X} : M_m \to M_m \otimes_\gamma B(X),\,[a_{ij}] \mapsto [a_{ij}\Id_X].
\]
Then $\sigma_{m,X}$ is matricially $\gamma$-bounded with $\sup_{m \in \N} \|\sigma_{m,X}\|_\mat < \infty$
if and only if $X$ has property $(\alpha).$
\end{prop}

It is shown in \cite[Lemma 4.3]{KrLM} that the $\sigma_{m,X}$ are $\gamma$-bounded uniformly in $m \in \N$ if and only if $X$ has property $(\alpha)$ (with Rademachers in place of Gaussians).
Actually the same proof applies to Proposition \ref{Prop sigma(m,X)}.

Mappings which are $\gamma$-bounded or matricially $\gamma$-bounded have been studied so far in connection with functional calculi and unconditional decompositions \cite{KrLM,dPR} where $E$ is a $C(K)$-space
and representations of amenable groups \cite{LM10}, where $E$ is a nuclear $C^*$-algebra.
We shall focus in this section on the row Hilbert space $E = \ell^2_r.$

\begin{thm}\label{Thm Main 1}
Let $u : \ell^2 \to B(X)$ be a bounded linear mapping.
Assume that $X$ has Pisier's property $(\alpha).$
% Note that for any $x \in X,\,u(\cdot)x \in B(\ell^2,X).$
For $n \in \N,$ denote by $C_n \subset M_n$ the subspace of matrices vanishing outside the first column.
Then the following conditions are equivalent:
\begin{enumerate}
\item $\bignorm{u(\cdot)x}_{\gamma(\ell^2,X)} \leq C \| x \|.$
\item $u : \ell^2_r \to B(X)$ is matricially $\gamma$-bounded.
\item The restriction $\Id \otimes u: C_n \otimes \ell^2_r \to M_n \otimes_\gamma B(X)$ is bounded uniformly in $n \in \N.$
\end{enumerate}
\end{thm}

\begin{proof}
We fix an orthonormal basis $(e_k)_k$ of $\ell^2.$
Write $T_k = u(e_k).$
Then condition (1) of the statement rewrites
\begin{equation}\label{Equ 1 Thm}
\bignorm{ u(\cdot)x }_{\gamma} = \sup_{n \in \N} \bignorm{ \sum_{k=1}^n \gamma_k \otimes T_k x }_{\Gauss_n(X)} \leq C \| x \|.
\end{equation}
On the other hand, for $[f_{ij}] \in M_n \otimes \ell^2,$ we have with $f^{(k)}_{ij} = \spr{f_{ij}}{e_k},$
\[\bignorm{(\Id_{M_n}\otimes u)[f_{ij}]}_{M_n \otimes_\gamma B(X)} = \lim_m \bignorm{ \left[ \sum_{k=1}^m f^{(k)}_{ij} T_k \right] }_{M_n \otimes_\gamma B(X)}.\]
Thus, condition (2) is equivalent to
\begin{equation}\label{Equ 2 Thm}
\bignorm{ \sum_{i=1}^n \sum_{k=1}^m \gamma_i \otimes f_{ij}^{(k)} T_k x_j }_{\Gauss_n(X)} \leq C \|[f_{ij}]\|_{M_n \otimes \ell^2_{m,r}} \bignorm{ \sum_{i=1}^n \gamma_i \otimes x_i }_{\Gauss_n(X)} ,
\end{equation}
where $f_{ij} = (f_{ij}^{(k)})_k \in \ell^2_m,$ and $C$ is independent of $n$ and $m.$
Denote the linear bounded mapping $\Gauss_n(X) \to \Gauss_n(X)$ arising from \eqref{Equ 2 Thm} by $u_{n,m}([f_{ij}]).$
Finally, condition (3) is equivalent to \eqref{Equ 2 Thm} with $f_{ij} = 0$ for $j \geq 2.$\\

\noindent
(1) $\Longrightarrow$ (2)

For $m \in \N$ fixed, let $Y = \Gauss_m(X)$ and define the operators
\[ V : X \to Y, \, x\mapsto \sum_{k=1}^m \gamma_k \otimes T_k x,\quad W : Y \to X,\, \sum_{k=1}^m \gamma_k \otimes x_k \mapsto x_1.\]
By assumption \eqref{Equ 1 Thm}, $V$ is bounded with constant $C$ independent of $m.$
Further, $W$ is bounded (see e.g. \cite[(2.13)]{Kr} for a simple proof).
For $n \in \N,$ denote
\[V_n = \Id_{\ell^2_n} \otimes V : \Gauss_n(X) \to \Gauss_n(Y),\, \sum_{k=1}^n \gamma_k \otimes x_k \mapsto \sum_{k=1}^n \gamma_k \otimes V(x_k).\]
It is easy to check that $\|V_n\| = \|V\|.$
Similarly, defining $W_n = \Id_{\ell^2_n} \otimes W : \Gauss_n(Y) \to \Gauss_n(X),$ one has $\|W_n\| = \|W\|.$
Let $i_m : \ell^2_{m,r} \to M_m$ be the first row identification as in \eqref{Equ i_m} which is completely bounded of $cb$-norm 1.
Then by Remark \ref{Rem mat} and Proposition \ref{Prop sigma(m,X)}, along with property $(\alpha),$ $\pi_m = \sigma_{m,X} \circ i_m : \ell^2_{m,r} \to B(M_m \otimes X)$ is a matricially $\gamma$-bounded mapping
and $\sup_{m \in \N} \|\pi_m\|_\mat < \infty.$
For $f = [f_{ij}]_{ij} \in M_n \otimes \ell^2_{m,r},$ we have the identity
\[ u_{n,m}(f) = \left[\sum_{k=1}^m f_{ij}^{(k)} T_k \right] = [W \pi_m(f_{ij}) V] = W_n [\pi_m(f_{ij})] V_n.\]
Therefore,
\[\|u_{n,m}(f)\| \leq \|W_n\| \, \|V_n\| \, \| [\pi_m(f_{ij})] \| \leq \|W\| \, \|V\| \, \|\pi_m\|_\mat \|[f_{ij}]\|_{M_n \otimes \ell^2_{m,r}},\]
so \eqref{Equ 2 Thm} follows.\\

\noindent
(2) $\Longrightarrow$ (3)

This is clear, since (3) is an obvious restriction of (2).\\

\noindent
(3) $\Longrightarrow$ (1)

Choose $n = m \in \N$ and $f = [f_{ij}] \in M_n \otimes \ell^2_{n,r}$ with $f_{ij} = \delta_{j1} e_i,$ where $(e_i)$ is the standard basis of $\ell^2_n.$
By definition of the row norm, we have
\begin{align*}
\|f\|_{M_n \otimes \ell^2_{n,r}} & = \bignorm{ \left[ \sum_k \spr{f_{ik}}{f_{jk}} \right] }^{\frac12}_{M_n} \\
& = \bignorm{ \left[ \sum_k \delta_{k1} \spr{e_i}{e_j} \right] }_{M_n}^{\frac12} \\
& = \bignorm{ [\delta_{ij}] }_{M_n}^{\frac12} \\
& = 1.
\end{align*}
As $f$ is supported by the first column, by assumption \eqref{Equ 2 Thm} there is some $C < \infty$ such that
\begin{align*}
C & \geq \|u_{n,n}(f)\|_{M_n \otimes_\gamma B(X)} \\
& = \bignorm{ \left(
                                                         \begin{array}{cccc}
                                                           T_1 & 0 & \ldots & 0 \\
                                                           \vdots & \vdots & \ddots & \vdots \\
                                                           T_n & 0 & \ldots & 0 \\
                                                         \end{array}
                                                       \right)}_{M_n \otimes_\gamma B(X)}
\\
& = \sup \left\{ \bignorm{ \sum_k \gamma_k \otimes T_k x_1 }_{\Gauss_n(X)} :\: \bignorm{ \sum_k \gamma_k \otimes x_k }_{\Gauss_n(X)} \leq 1 \right\} \\
& = \sup \left\{ \bignorm{ \sum_k \gamma_k \otimes T_k x   }_{\Gauss_n(X)} :\: \|x\| \leq 1 \right\}.
\end{align*}
Letting $n \to \infty$ shows that \eqref{Equ 1 Thm} holds.
\end{proof}

\begin{rem}
Theorem \ref{Thm Main 1} is a generalization of \cite[Proposition 3.3]{LM04}, where $X$ is an $L^p$-space, and \cite[Corollary 3.19]{HaKu}, where more generally $X$ has property $(\alpha).$
There it is shown that condition (1) of the theorem implies that $u : \ell^2 \to B(X)$ is $\gamma$-bounded.
(In these two references, $u$ maps to $B(Y,X)$ instead of $B(X).$
A corresponding version of Theorem \ref{Thm Main 1} with $B(Y,X)$ in place of $B(X)$ holds with the same proof).
\end{rem}

\section{The H\"ormander functional calculus}\label{Sec 4 Hormander operator space}

Recall the spaces $\Wa$ and $\Ha$ and the dyadic partition of unity $(\phi_n)_{n \in \Z}$ from the introduction.
Clearly the space $\Wa$ is a Hilbert space.
We equip $\Ha$ with an operator space structure by putting
\begin{equation}\label{Equ Ha OSS}
\|[f_{ij}]\|_{M_n \otimes \Ha} = \sup_{k \in \Z} \| [\phi_k f_{ij}] \|_{M_n \otimes \Wa_r},
\end{equation}
where the index $r$ refers to the row space structure.
It is easy to check that \eqref{Equ Ha OSS} indeed defines an operator space, arising from the embedding
\[ \Ha \hookrightarrow B(\bigoplus^2_{k \in \Z} \Wa),\, f \mapsto \left((g_k)_{k \in \Z} \mapsto (\spr{g_k}{\phi_k f} e)_{k \in \Z} \right),\]
where $\displaystyle \bigoplus^2_{k \in \Z} \Wa$ is the Hilbert sum, and $e$ is some fixed element in $\Wa$ of norm 1.

In this section we focus on (unital) homomorphisms
\begin{equation}\label{Equ u} u : \Ha \to B(X) .
\end{equation}
We give a characterization when such mappings are matricially $\gamma$-bounded.
An example of a bounded homomorphism of this type is given by H\"ormander's classical theorem mentioned in the introduction,
which states that for $\alpha > \frac{d}{2},\: X = L^p(\R^d)$ and $p \in (1,\infty),$ the radial Fourier multiplier representation
$u_{-\Delta} : \Ha \to B(X) $
given by 
\begin{equation}\label{Equ u Delta}
u_{-\Delta}(f)g = \left[ f(|\cdot|^2) \hat{g} \right]\check{\phantom{i}} = f(-\Delta) g
\end{equation}
is bounded.
In fact, by means of our characterization, we will show in Section \ref{Sec 5 Examples} that $u_{-\Delta}$ is even matricially $\gamma$-bounded provided that $\alpha > \frac{d+1}{2}.$

For $n \in \N,$ let $\M^n$ be the space consisting of $n$-times continuously differentiable functions $f$ defined on $(0,\infty)$ such that
$\|f\|_{\M^n} = \sum_{k=0}^n \sup_{t > 0} |t^k f^{(k)}(t)|$ is finite.
Let us record how $\Wa,\,\Ha,\,\HI(\Sigma_\omega)$ and the auxiliary space $\M^n$ compare.
The proof is an easy verification, see also \cite[Lemma 4.15, Proof of Proposition 4.22, Proposition 4.9]{Kr}.

\begin{lem}\label{Lem Function spaces}
Let $\omega \in (0,\pi)$ and $\alpha > \frac12.$
\begin{enumerate}
\item $\HI(\Sigma_\omega) \hookrightarrow \Ha.$
\item $\Wa \hookrightarrow \Ha,$ where the embedding is completely bounded.
\item For any $\omega \in (0,\pi),$ $\HI(\Sigma_\omega) \cap \Wa$ is a dense subset of $\Wa.$
\item $\M^n \hookrightarrow \Ha$ for $n > \alpha.$
\item $\HI(\Sigma_\omega)$ is a dense subset of $\M^n.$
Moreover, any $f \in \Wa \cap \M^n$ can be simultaneously approximated by a sequence $(f_k)_{k \in \N} \subset \HI(\Sigma_\omega) \cap \Wa \cap \M^n.$
\end{enumerate}
\end{lem}

The main interest of $\M^n$ is the following convergence lemma, which is proved in \cite[Section 4.2.4]{Kr}.

\begin{lem}\label{Lem Convergence Mn}
Let $u : \M^n \to B(X)$ be bounded such that $u(f) = f(A)$ for some $0$-sectorial operator $A$ and any $f \in \bigcup_{\theta \in (0,\pi)} \HI(\Sigma_\theta).$
Let $(\phi_n)_{n \in \Z}$ be a dyadic partition of unity and $(a_n)_{n \in \Z}$ a bounded sequence.
Then $\sum_{n \in \Z} a_n \phi_n$ belongs to $\M^n$ and
\begin{equation}\label{Equ Lem A u}\sum_{n \in \Z} a_n u(\phi_n) x = u \left( \sum_{n \in \Z} a_n \phi_n \right)x \quad (x \in X). \end{equation}
\end{lem}

Many spectral multiplier theorems for Laplace type operators $A$ consist in the boundedness of $u$ in \eqref{Equ u}, which in turn is the functional calculus $u_A$ of some $0$-sectorial operator.
For example, in the case of \eqref{Equ u Delta} one has $A = - \Delta.$
In the sequel we will only consider homomorphisms of the form $u = u_A.$
The next lemma gives a criterion when this is the case.

\begin{lem}\label{Lem A u}
Let $\omega \in (0,\pi).$
\begin{enumerate}
\item
Let $u$ be a bounded homomorphism $u : \Ha \to B(X).$
There exists a $0$-sectorial operator $A$ such that
\begin{equation}\label{Equ Lem A u 2}
u(f) = f(A) \quad (f \in \HI(\Sigma_\omega)),
\end{equation}
if and only if the restriction of $u$ to $\HI(\Sigma_\omega)$ satisfies the Convergence Lemma \ref{Lem Convergence Lemma},
i.e. for any $f_n,\: f \in \HI(\Sigma_\omega)$ with $\sup_n \|f_n\|_{\infty,\omega} < \infty$ and $f_n(\lambda) \to f(\lambda)$ pointwise, we have $u(f)x = \lim_n u(f_n)x$ for any $x \in X.$
In this case, we write $f(A)$ in place of $u(f)$ for any $f \in \Ha.$
\item
Let $A$ be a $0$-sectorial operator.
Then there exists a bounded homomorphism $u : \Ha \to B(X)$ satisfying \eqref{Equ Lem A u 2} if and only if $\|f(A)\| \leq C \|f\|_{\Ha}\quad (f \in \HI(\Sigma_\omega)).$
Moreover, $u$ is uniquely determined by \eqref{Equ Lem A u 2}.
\end{enumerate}
\end{lem}

\begin{proof}
(1) By Lemma \ref{Lem Function spaces} (1), only the ``if'' part has to be shown.
Suppose that $u : \Ha \to B(X)$ satisfies the Convergence Lemma \ref{Lem Convergence Lemma}.
Since $u$ is a homomorphism, one has $u((\lambda - \cdot)^{-1}) - u((\mu - \cdot)^{-1}) = (\mu - \lambda) u((\lambda - \cdot)^{-1}) u((\mu - \cdot)^{-1})$
for any $\lambda,\mu \in \C \backslash [0,\infty),$ i.e. $u((\lambda - \cdot)^{-1})$ is a pseudo resolvent \cite[Definition 9.1]{Paz}.
By Lemma \ref{Lem Convergence Lemma}, $u(\lambda (\lambda - \cdot)^{-1} )x \to x$ for any $x \in X$ and $|\lambda|\to \infty.$
Thus, by \cite[Corollary 9.5]{Paz}, there exists a densely defined operator $A$ such that $u((\lambda - \cdot)^{-1}) = (\lambda-A)^{-1}$ for $\lambda \in \C \backslash [0,\infty).$
Again by Lemma \ref{Lem Convergence Lemma}, $-\frac1n (\frac1n + A)^{-1}x \to 0$ for any $x \in X$ and $n \to \infty.$
Thus, $A$ has dense range \cite[Proposition 15.2]{KuWe}.
As $u$ is a bounded homomorphism, it now follows that $A$ is $0$-sectorial and for any rational function $r \in \Ha,$
$\|r(A)\| \lesssim \|r\|_{\Ha} \lesssim \|r\|_{\infty,\omega}.$

We claim that $A$ has an $\HI$ calculus coinciding with $u.$
Indeed, a given $f \in \HI_0(\Sigma_\omega),$ we write
\[f = \frac{1}{2\pi i} \int_{\partial\Sigma_{\omega/2}} f(\lambda) (\lambda - \cdot)^{-1} d\lambda.\]
As $f(\lambda) (\lambda - \cdot)^{-1} : \partial\Sigma_{\omega/2} \to \HI(\Sigma_{\omega/4})$ is continuous, we find a sequence
$r_n = \sum_{k=1}^K c_k f(\lambda_k) (\lambda_k-\cdot)^{-1}$ such that $r_n \to f$ in $\HI(\Sigma_{\omega/4}),$ so in particular in $\Ha.$
Clearly, $r_n$ are rational functions.
Inserting formally $(\cdot) = A$ in the Cauchy integral, the same arguments apply, and $r_n(A) \to f(A).$
We conclude $u(f) = \lim_n u(r_n) = \lim_n r_n(A) = f(A).$

We have shown that $u(f) = f(A)$ for any $f \in \HI_0(\Sigma_\omega).$
For a general $f \in \HI(\Sigma_\omega)$ we use the approximation \eqref{Equ Convergence lemma}.\\

(2) The ``only if'' part is clear and  the ``if'' part is shown in \cite[Remark 4.27]{Kr}.
Using density and Lemma \ref{Lem Function spaces}, $u$ is uniquely determined on $\Wa$ and $\M^n$ for any $n > \alpha.$
Thus $u$ satisfies the decomposition \eqref{Equ Lem A u}.
Then for $f \in \Ha,$ we have $u(f)x  = u(f) \sum_{k \in \Z} u(\phi_k) x = \sum_{k \in \Z} u(f \phi_k) x.$
As $f \phi_k \in \Wa,$ we conclude the uniqueness of $u.$
\end{proof}

The strategy to prove matricial $\gamma$-boundedness of a mapping from $\Ha$ to $B(X)$ will be to show the matricial $\gamma$-boundedness from $\Wa$ to $B(X),$
and then to pass to $\Ha$ by means of a spectral decomposition, given by \eqref{Equ Paley Littlewood} in the following theorem.
The restriction of the $\HI$ calculus angle $\omega$ to $(0,\pi/4)$ is only for technical reasons.

\begin{thm}\label{Thm Paley Littlewood}
Let $X$ be a Banach space with property $(\alpha).$
Let $\alpha > \frac12,\,\omega \in (0,\pi/4)$ and $A$ be a $0$-sectorial operator on $X$ having a bounded $\HI(\Sigma_\omega)$ calculus.
Assume that \[\|f(A)\| \leq C \|f\|_{\Wa} \quad (f \in \HI(\Sigma_\omega) \cap \Wa) \]
and that the extension resulting from density $u : \Wa \to B(X),\, f \mapsto f(A)$ is matricially $\gamma$-bounded.
Let $(\phi_n)_{n \in \Z}$ be a dyadic partition of unity.
Then
\begin{equation}\label{Equ Paley Littlewood}
\|x\| \cong \| \sum_{n \in \Z} \gamma_n \otimes \phi_n(A)x \|_{\Gauss(X)},
\end{equation}
the sum on the right hand side being convergent in $\Gauss(X).$
\end{thm}

As the proof is rather long we separate four preliminary lemmas, whose proofs are annexed in Section \ref{Sec 6 Proofs Lemmas}.

\begin{lem}\label{Lem 1}
Let $X$ have property $(\alpha)$, let $\alpha > \frac12.$
Let $A$ be as in Theorem \ref{Thm Paley Littlewood}.
Then for $\beta \in \N,\,\beta > \alpha,$
\begin{equation}\label{Equ NT}
\left\{\exp(-2^k z A) :\: k \in \Z \right\} \text{ is }\gamma\text{-bounded with constant }\lesssim \left| \frac{z}{\Re z}\right|^\beta \quad (\Re z > 0).
\end{equation}
\end{lem}

\begin{lem}\label{Lem 2}
Let $A$ be a $0$-sectorial operator on some Banach space $X$ such that for some $\beta > 0,$ \eqref{Equ NT} holds.
Then for $\gamma = \beta + 1,$ we have
\begin{equation}\label{Equ Res}
\left\{ \lambda^{\frac12} (2^k A)^{\frac12}(\lambda - 2^k A)^{-1} :\: k \in \Z \right\}\text{ is }\gamma\text{-bounded with constant }\lesssim
|\arg \lambda|^{-\gamma} \quad (\Re \lambda > 0).
\end{equation}
\end{lem}

\begin{lem}\label{Lem 3}
Let $A$ be a $0$-sectorial operator on some space $X$ with property $(\alpha)$ having a bounded $\HI(\Sigma_\omega)$ calculus.
Assume that $A$ satisfies \eqref{Equ Res} for some $\gamma > 0.$
Then for any $n > \gamma,$
\[ \|f(A)\| \leq C \|f\|_{\M^n} \quad \left(f \in \HI(\Sigma_\omega)\right).\]
\end{lem}

\begin{lem}\label{Lem 4}
Let $n \in \N.$
Let $(g_k)_{k \in \Z}$ satisfy $\sup_{k \in \Z} \|g_k\|_{\M^n} < \infty.$
Suppose that the supports of $g_k$ satisfy the following overlapping condition
\[ \sup_{x > 0}\#\{ k \in \Z:\: \supp g_k \cap [\frac12 x,2x] \not= \emptyset \} < \infty.\]
Then $\sum_{k \in \Z} g_k,$ which is consequently pointwise a finite sum belongs to $\M^n,$ and
\[\|\sum_{k \in \Z} g_k\|_{\M^n} \lesssim \sup_{k\in\Z} \|g_k\|_{\M^n} < \infty.\]
\end{lem}

\begin{proof}[Proof of Theorem \ref{Thm Paley Littlewood}]
Using Lemmas \ref{Lem 1}, \ref{Lem 2} and \ref{Lem 3} one after another shows that $\|f(A)\| \lesssim \|f\|_{\M^n}$ for $n$ sufficiently large $(n > \lfloor \alpha \rfloor + 2).$
For any $k \in \Z,$ let $a_k \in \{ 1 , - 1\}.$
Apply Lemma \ref{Lem 4} with $g_k = a_k \phi_k.$
It is easy to check that $\|g_k\|_{\M^n}$ is independent of $k \in \Z.$
Further, the overlapping condition is clearly satisfied with constant 2.
Thus we have, for any finite $F \subset \Z,$
\[ \| \sum_{k \in F} a_k \phi_k(A) x\| \lesssim \|\sum_{k \in F} a_k \phi_k \|_{\M^n} \|x\| \lesssim \|x\|.\]
Replacing $a_k$ by independent Rademacher variables $\epsilon_k$ and taking expectation gives
\[\|\sum_{k \in F} \epsilon_k \otimes \phi_k(A) x\|_{\Rad(X)} \lesssim \|x\|.\]
Since $X$ has property $(\alpha),$ the equivalence of Gaussian and Rademacher sums \eqref{Equ Gauss Rad} holds.
By \eqref{Equ Lem A u}, $\sum_{k \in \Z} a_k \phi_k(A)x$ converges in $X.$
By dominated convergence (resp. \eqref{Equ Gauss Rad}), convergence holds also in $\Rad(X)$ (resp. $\Gauss(X)$), when $a_k$ is replaced by $\epsilon_k$ (resp. $\gamma_k$).

We have shown
\[ \| \sum_{k \in \Z} \gamma_k \otimes \phi_k(A) x \|_{\Gauss(X)} \lesssim \|x\|. \]

For the reverse inequality, we argue by duality.
Let $x' \in X',$ write $\E$ for expectation and $\widetilde\phi_l = \sum_{k=l-1}^{l+1}\phi_k.$
By the support condition on the $\phi_k,$ $\widetilde\phi_l \phi_l = \phi_l.$
Then using the independence of the $\gamma_k,$ we have
\begin{align*}
\left|\spr{x}{x'}\right| & = \left|\E \spr{\sum_{k \in \Z} \gamma_k \phi_k(A)x}{\sum_{l \in \Z} \gamma_l \widetilde\phi_l(A)'x'}\right| \\
& \leq \|\sum_{k \in \Z} \gamma_k \otimes \phi_k(A) x\|_{\Gauss(X)} \| \sum_{l \in \Z} \gamma_l \otimes \widetilde\phi_l(A)'x'\|_{\Gauss(X')}.
\end{align*}
We conclude the proof by the same argument as above which shows that
\[\| \sum_{l \in \Z} \gamma_l \otimes \widetilde\phi_l(A)' x'\|_{\Gauss(X')} \lesssim \|x'\|.\]
\end{proof}

The main result of this section reads as follows.

\begin{thm}\label{Thm Ha}
Let $A$ be a $0$-sectorial operator with bounded $\HI(\Sigma_\theta)$ calculus for some $\theta \in (0,\pi)$ on a space $X$ with property $(\alpha).$
Let $\alpha > \frac12.$
Then the following are equivalent.
\begin{enumerate}
\item For any $x \in X,$ $(1+t^2)^{-\alpha/2} A^{it}x$ belongs to $\gamma(\R,X)$ and
\[ \| (1 + t^2)^{-\alpha/2} A^{it} x\|_{\gamma(\R,X)} \leq C \|x\|.\]
\item The $\HI$ calculus of $A$ extends to a matricially $\gamma$-bounded mapping $u : \Ha \to B(X).$
\end{enumerate}
\end{thm}

\begin{proof}
Assume first that the theorem is shown under the additional assumption that $\theta < \frac{\pi}{4}.$
For a general $\theta \in (0,\pi),$ we can reduce to this case by considering $B = A^{\frac14}.$
Namely, by \cite[Theorem 2.4.2]{Haa}, $B$ has a bounded $\HI(\Sigma_\omega)$ calculus for some $\omega \leq \frac{\theta}{4} < \frac{\pi}{4}.$
Moreover,
\[ (1 + t^2)^{-\alpha/2} B^{it} x  = \left( \frac{1 + (\frac{t}{4})^2}{1 + t^2}\right)^{\alpha/2} \cdot \left( 1 + \frac{t}{4} \right)^{-\frac{\alpha}{2}} A^{i\frac{t}{4}} x.\]
The first factor is bounded, so its multiplication with an $L^2(\R)$ function is a bounded operation on $L^2(\R).$
The same holds for its inverse, and also for the change of variables $f \mapsto f(\frac{\cdot}{4}),$ and its inverse.
Thus, by Lemma \ref{Lem Folklore square functions}, if $A$ satisfies (1) then so does $B,$ so $B$ satisfies (2).
As $\HI(\Sigma_\omega) \hookrightarrow \Ha,$ $B$ has an $\HI(\Sigma_\omega)$ calculus actually for any $\omega >0,$ so by \cite[Theorem 2.4.2]{Haa}, $A$ has an $\HI(\Sigma_\theta)$ calculus for some $\theta < \frac{\pi}{4}.$
The same is true, provided that $A$ satisfies (2).
Thus (1) or (2) imply that the assumption of the theorem actually holds with $\theta < \frac{\pi}{4}.$
We suppose from now on that $\theta < \frac{\pi}{4}.$\\

\noindent
$(1) \Longrightarrow (2).$\\
Fix an arbitrary orthonormal basis $(f_k)_k$ of $L^2(\R).$
Let $T_k \in B(X)$ be defined by
\[\spr{T_k x}{x'} = \frac{1}{2 \pi} \int_\R f_k(t) (1 + t^2)^{-\alpha/2} \spr{A^{it} x}{x'} dt\]
and $v : \ell^2 \to B(X)$ the linear mapping given by $e_k \mapsto T_k.$
Then $(1)$ implies
\[\|v(\cdot)x\|_{\gamma(\ell^2,X)} = \| \sum_k \gamma_k \otimes v(e_k)x \|_{\Gauss(X)} = \| \sum_k \gamma_k \otimes T_k x \|_{\Gauss(X)} = \|(1 + t^2)^{-\frac{\alpha}{2}} A^{it}x\|_\gamma \lesssim \|x\|,\] so that by Theorem \ref{Thm Main 1}, $v : \ell^2_r \to B(X)$ is matricially $\gamma$-bounded.
Consider the mapping $w : \Wa \to \ell^2,\, f \mapsto \left( \frac{1}{2 \pi} \int_\R (f\circ \exp)\hat{\phantom{i}}(t) (1 + t^2)^{\alpha/2} \overline{f_k}(t) dt \right)_k.$
It is easy to check that $w$ is unitary and consequently, $\tilde{u} = v \circ w : \Wa_r \to B(X)$ is also matricially $\gamma$-bounded.
On the other hand, $\tilde{u}(f) = f(A)$ for any $f \in \HI(\Sigma_\theta) \cap \Wa.$
Indeed, by the representation formula in \cite[Proposition 4.22]{Kr},
\begin{align*}
2 \pi \spr{f(A)x}{x'} & = \int_\R (f \circ \exp)\hat{\phantom{i}}(t) \spr{A^{it}x}{x'} dt\\
& = \int_\R (f \circ \exp)\hat{\phantom{i}}(t) (1 + t^2)^{\alpha/2} (1 + t^2)^{-\alpha/2} \spr{A^{it}x}{x'} dt\\
& = \sum_{k} \int_\R (f\circ\exp)\hat{\phantom{i}}(t) (1 + t^2)^{\alpha/2} \overline{f_k}(t) dt \spr{T_k x}{x'}\\
& = 2 \pi \spr{\tilde{u}(f)x}{x'}.
\end{align*}

Let $n \in \N$ and $F = [f_{ij}] \in M_n \otimes \HI(\Sigma_\theta).$
We show that $\| [ f_{ij}(A) ] \|_{M_n \otimes_\gamma B(X)} \lesssim \|[f_{ij}]\|_{M_n \otimes \Ha}.$
For $N \in \N$ consider
\[ F_N = \left[
           \begin{array}{cccc}
             [\phi_{-N} f_{ij}] & 0 & \ldots & 0 \\
             0 & [\phi_{-N + 1} f_{ij}] &  & \vdots \\
             \vdots &  & \ddots & 0 \\
             0 & \ldots & 0 & [\phi_N f_{ij}] \\
           \end{array}
         \right] \in M_{(2N + 1)n} \otimes \Wa.
\]
By \eqref{Equ Ha OSS}, we have $\sup_N \|F_N\|_{M_{(2N+1)n} \otimes \Wa} = \|F\|_{M_n \otimes \Ha}.$
Observe first that for any scalars $g_1,\ldots,g_n,$ by Theorem \ref{Thm Paley Littlewood}, with $\tilde\phi_k = \sum_{l = k - 1}^{k + 1} \phi_l,$
\[\| \sum_{i,j = 1}^n g_i f_{ij}(A)x_j \| \cong \| \sum_{i,j = 1}^n \sum_{k \in \Z} \gamma_k \otimes g_i \tilde{u}(f_{ij} \phi_k) x_j \|
\cong \| \sum_{i,j = 1}^n \sum_{k \in \Z} \gamma_k \otimes g_i \tilde{u}(f_{ij} \phi_k) \tilde{u}(\widetilde\phi_k) x_j \|.\]
Replacing $g_i$ by Gaussian variables and taking expectations shows that
\begin{equation}\label{Equ Proof Main Thm 1}
\|\sum_{i,j=1}^n \sum_{k \in \Z} \gamma_i \otimes \tilde{u}(f_{ij}) x_j\|_{\Gauss_n(X)} \cong \| \sum_{i,j=1}^n \sum_{k \in \Z} \gamma_i \otimes \gamma_k \otimes
\tilde{u}(f_{ij}\phi_k) \tilde{u}( \widetilde\phi_k) x_j\|_{\Gauss_n(\Gauss(X))}.
\end{equation}
Further we have
\begin{align}
& \| \sum_{i,j=1}^n \sum_{k=-N}^N \gamma_i \otimes \gamma_k \otimes \tilde{u}(f_{ij} \phi_k) \tilde{u}(\widetilde\phi_k) x_j \|_{\Gauss_n(\Gauss(X))} \nonumber \\
& \cong \| \sum_{i,j,k} \gamma_{ik} \otimes \tilde{u}(f_{ij} \phi_k) \tilde{u}(\widetilde\phi_k) x_j\|_{\Gauss(X)} \nonumber \\
& \lesssim \| F_N \|_{M_{(2N+1)n} \otimes \Wa} \| \sum_{i,k} \gamma_{i,k} \otimes \tilde{u}(\widetilde\phi_k) x_i \|_{\Gauss(X)} \nonumber \\
& \lesssim \|F\|_{M_n \otimes \Ha} \|\sum_i \gamma_i \otimes x_i \|_{\Gauss_n(X)}. \label{Equ Proof Main Thm 2}
\end{align}
Finally taking the supremum over $N \in \N,$ \eqref{Equ Proof Main Thm 1} and \eqref{Equ Proof Main Thm 2} give
\[ \| [f_{ij}(A)] \|_{M_n \otimes_\gamma B(X)} \lesssim \| [f_{ij}] \|_{M_n \otimes \Ha} .\]
In particular, $\|f(A)\| \lesssim \|f\|_{\Ha},$ so that by Lemma \ref{Lem A u}, there exists a bounded mapping $u : \Ha \to B(X)$ extending the $\HI$ calculus in the sense of \eqref{Equ Lem A u 2}.
Now repeat the above argument with an arbitrary $F = [f_{ij}] \in M_n \otimes \Ha$ and $u(f_{ij})$ in place of $f_{ij}(A),$ to deduce that $u$ is matricially $\gamma$-bounded.\\

\noindent
$(2) \Longrightarrow (1).$
Denote $\tilde{u}$ the restriction of $u$ to $\Wa$ which by Lemma \ref{Lem Function spaces} is again matricially $\gamma$-bounded.
Thus also the mapping $v= \tilde{u} \circ w^{-1}$ from the first part of the proof is matricially $\gamma$-bounded and by Theorem \ref{Thm Main 1},
$\| (1 + t^2)^{-\alpha/2} A^{it} x \|_{\gamma(\R,X)} \leq C \|x\|.$
\end{proof}

\section{Extensions and Applications}\label{Sec 5 Examples}

We have characterized in Theorem \ref{Thm Ha} the matricially $\gamma$-bounded H\"ormander calculus in terms of square functions of $A.$
In fact, the imaginary powers $A^{is}$ appearing in these square functions can be replaced by several other typical operator families associated with $A,$
such as resolvents $R(\lambda,A)$ for $\lambda \in \C \backslash [0,\infty)$ and the semigroup generated by $-A,$ $T(z) = \exp(-zA)$ for $\Re z > 0.$
This gives (almost) equivalent conditions, see Proposition \ref{Prop BIP Sgr Res} below.
Subsequently, we use the semigroup condition of this proposition to apply Theorem \ref{Thm Ha} to some examples.
The starting point for us will be semigroups that satisfy (generalized) Gaussian estimates (see \eqref{Equ Gaussian estimate}).

The following lemma serves as a preparation for Proposition \ref{Prop BIP Sgr Res}.

\begin{lem}\label{Lem square function integral transform}
For $i = 1,2,$ let $(\Omega_i,\mu_i)$ be $\sigma$-finite measure spaces and $K \in B(L^2(\Omega_1),L^2(\Omega_2)).$
\begin{enumerate}
\item
Assume that $f \in \gamma(\Omega_1,X)$ and that there exists a Bochner-measurable $g : \Omega_2 \to X$ such that
\[ \spr{g(\cdot)}{x'} = K(\spr{f(\cdot)}{x'}) \quad (x' \in X'). \]
Then $g \in \gamma(\Omega_2,X)$ and
\[ \|g\|_{\gamma(\Omega_2,X)} \leq \|K\|\,\|f\|_{\gamma(\Omega_1,X)}. \]
\item
Let $\Omega_1 \to B(X),\,t \mapsto N(t)$ and $\Omega_2 \to B(X),\,t \mapsto M(t)$ be weakly measurable.
Assume that $\|N(\cdot)x\|_\gamma \leq C \|x\|$ and that there is $K \in B(L^2(\Omega_1),L^2(\Omega_2))$ such that
$K \left[ \spr{N(\cdot)x}{x'} \right] = \spr{M(\cdot)x}{x'}$ for $x \in D,$ where $D$ is some dense subset of $X.$
Then $M(\cdot)x \in \gamma(\Omega_2,X)$ for any $x \in X$ and $\|M(\cdot)x\|_\gamma \lesssim \|N(\cdot)x\|_\gamma.$
\item
Let $(\Omega,\mu)$ be a measure space and $g :\Omega \to X$ measurable.
For $n \in \N,$ let $\varphi_n:\Omega \to [0,1]$ measurable with $\sum_{n = 1}^\infty \varphi_n(t) = 1$ for all $t \in \Omega.$
Then
\[ \|g\|_{\gamma(\Omega,X)} \leq \sum_{n = 1}^\infty \| \varphi_n g \|_{\gamma(\Omega,X)}. \]
\end{enumerate}
\end{lem}

\begin{proof}
(1) By assumption, $g \in P_2(\Omega_2,X).$
Consider the associated operator $u_g : H \to X$ as in \eqref{Equ u_f}.
We have $u_g = u_f \circ K'.$
Thus, by Lemma \ref{Lem Folklore square functions}, $\|u_g\|_{\gamma(L^2(\Omega_2),X)} \leq \|K'\|\,\|u_f\|_{\gamma(L^2(\Omega_1),X)},$ which proves (1).\\

\noindent
(2) We show first that $M(\cdot)x$ belongs to $P_2(\Omega_2,X)$ for any $x \in X.$
For $x \in D,$ this follows immediately from the assumption.
For $x \in X,$ we let $x_n \in D$ such that $x_n \to x$ $(n \to \infty).$
Then $\spr{M(t)x}{x'} = \lim_n \spr{M(t)x_n}{x'}$ for any $t \in \Omega_2.$
On the other hand, $\spr{M(\cdot)x_n}{x'}$ is convergent in $L^2(\Omega_2).$
Indeed,
\begin{align*}
\| \spr{M(\cdot)(x_n - x_m)}{x'} \|_{L^2(\Omega_2)} & = \|K\left[ \spr{N(\cdot)(x_n - x_m)}{x'} \right] \|_{L^2(\Omega_1)} \\
& \lesssim \| N(\cdot)(x_n - x_m)\|_\gamma \, \|x'\| \lesssim \|x_n - x_m\| \, \|x'\|,
\end{align*}
which converges to $0\: (n,m \to \infty).$
Thus, $\spr{M(\cdot)x}{x'}$ the pointwise limit, so necessarily equal to the $L^2$ limit, belongs to $L^2(\Omega_2).$
Consequently, by (1) and Lemma \ref{Lem Folklore square functions}, $\|M(\cdot)x\|_\gamma = \lim_n \|M(\cdot)x_n\|_\gamma \leq \|K\| \lim_n \|N(\cdot)x_n\|_\gamma = \|K\| \,\|N(\cdot)x\|_\gamma,$ which shows (2).\\

\noindent
(3) For $n \in \N,$ put $\phi_n = \sum_{k=1}^n \varphi_k.$
Then $\phi_n: \Omega \to [0,1]$ and $\phi_n(t) \to 1$ monotonically for all $t \in \Omega.$
Then $\sup_n \|\phi_n g\|_\gamma \leq \sup_n \sum_{k=1}^n \| \varphi_k g\|_\gamma = \sum_{k=1}^\infty \|\varphi_k g \|_\gamma.$
It remains to show $\|g\|_\gamma \leq \sup_n \| \phi_n g\|_\gamma.$
Let us show first that $g \in P_2(\Omega,X),$ i.e. for any $x' \in X',\, \spr{g(\cdot)}{x'} \in L^2(\Omega).$
By assumption, we have $|\spr{g(t)}{x'}| = \lim_n \phi_n(t) |\spr{g(t)}{x'}|$ for any $t \in \Omega,$ and this convergence is monotone.
Then by Beppo Levi's theorem,
\[\|\spr{g(\cdot)}{x'}\|_{L^2(\Omega)} = \lim_n \|\spr{\phi_n(\cdot) g(\cdot)}{x'}\|_{L^2(\Omega)}
\leq \limsup_n \|\phi_n \cdot g\|_{\gamma(\Omega,X)} \|x'\|.\]
where we have used that $\|\spr{f(\cdot)}{x'}\|_{L^2(\Omega)} \leq \|f\|_{\gamma(\Omega,X)} \, \|x'\|$ for any $f \in \gamma(\Omega,X).$
Thus we have shown that $g \in P_2(\Omega,X).$
Then by Lemma \ref{Lem Folklore square functions},
\[\|g\|_{\gamma(\Omega,X)} \leq \liminf_n \|\phi_n \cdot g\|_{\gamma(\Omega,X)} \leq \sup_n \|\phi_n \cdot g\|_{\gamma(\Omega,X)}.\]
\end{proof}

\begin{prop}\label{Prop BIP Sgr Res}
Let $A$ be a $0$-sectorial operator having a bounded $\HI$ calculus on some space $X$ with property $(\alpha).$
Let $\alpha > \frac12.$
Consider the following conditions.
\[
\begin{array}{ll}
& \textit{H\"ormander functional calculus}\\
(1) & \text{The }\HI\text{ calculus of }A\text{ extends to a matricially }\gamma\text{-bounded mapping }\Ha \to B(X).\vspace{0.2cm}\\
& \textit{Imaginary powers}\\
(2) &\| (1 + t^2)^{-\alpha/2} A^{it} x \|_{\gamma(\R,X)} \leq C \|x\|.\vspace{0.2cm}\\
& \textit{Resolvents}\\
(3)&\text{For some }\beta \in (0,1)\text{ and }\theta \in (-\pi,\pi) \backslash \{ 0 \} :\: \| t^\beta A^{1 - \beta} R(e^{i\theta} t,A) x \|_{\gamma(\R_+,dt/t,X)} \lesssim |\theta|^{-\alpha} \|x\|.\\
(4)&\text{For some }\beta \in (0,1),\:\theta_0 \in (0,\pi] :\: \| \, \|\theta|^{\alpha - \frac12} t^\beta A^{1 - \beta}R(e^{i \theta} t,A)x \|_{\gamma(\R_+ \times [-\theta_0,\theta_0], dt/t d\theta,X)} \lesssim \|x\|.\vspace{0.2cm}\\
& \textit{Analytic semigroup}\\
(5)&\text{For }\theta \in (-\frac{\pi}{2},\frac{\pi}{2}),\:\|A^{1/2}T(e^{i\theta}t)x\|_{\gamma(\R_+,dt,X)} \lesssim (\frac{\pi}{2}-|\theta|)^{-\alpha} \|x\|.\\
(6)& \|( 1+ \left| \frac{a}{b} \right|^2)^{\alpha/2} |a|^{-\frac12} A^{1/2} T(a+ib)x\|_{\gamma(\R_+\times \R,da db,X)} \lesssim \|x\|.
\end{array}
\]
Then the conditions (1), (2), (4), (6) are equivalent.
Further these conditions imply the remaining ones (3), (5), which conversely imply that the $\HI$ calculus of $A$ extends to a matricially $\gamma$-bounded homomorphism $\Hae \to B(X)$ for any $\epsilon > 0.$
\end{prop}

\begin{proof}
$(1) \Longleftrightarrow (2).$\\
This is Theorem \ref{Thm Ha}.\\

\noindent
$(2) \Longleftrightarrow (4).$\\
Consider
\begin{equation}\label{Equ Proof Charact K}
 K : L^2(\R,ds) \to L^2(\R \times (-\pi,\pi),ds d\theta),\,
f(s) \mapsto (\pi-|\theta|)^{\alpha-\frac12}\frac{1}{\sin \pi(\beta + is)} e^{\theta s} \langle s \rangle^{\alpha} f(s),
\end{equation}
where we write in short
\[ \langle s \rangle = (1 + |s|^2)^{\frac12}.\]
Note that $|\sin \pi(\beta + is)| \cong \cosh(\pi s)$ for $\beta \in (0,1)$ fixed.
$K$ is an isomorphic embedding.
Indeed,
\[ \|Kf\|_2^2 =
\int_\R \int_{-\pi}^\pi \left((\pi-|\theta|)^{\alpha-\frac12} e^{\theta s}\right)^2 d\theta \frac{1}{|\sin^2(\pi(\beta + i s))|} \langle s \rangle^{2\alpha} |f(s)|^2 ds\]
and
\begin{align}
\int_{-\pi}^\pi (\pi-|\theta|)^{2\alpha -1} e^{2\theta s} d\theta
& \cong \int_0^\pi \theta^{2\alpha -1} e^{2(\pi - \theta) |s|} d\theta
\nonumber \\
& \cong \cosh^2(\pi s) \int_0^\pi \theta^{2\alpha-1} e^{2\theta |s|} d\theta.
\nonumber
\end{align}
The last integral is bounded from below uniformly in $s \in \R,$ and for $|s|\geq 1,$
\[\int_0^\pi \theta^{2\alpha -1} e^{2\theta |s|} d\theta = (2|s|)^{-2\alpha} \int_0^{2|s|\pi} \theta^{2\alpha -1} e^{\theta}d\theta \cong |s|^{-2\alpha}.\]
This clearly implies that $\|Kf\|_2 \cong \|f\|_2.$
Applying Lemma \ref{Lem square function integral transform}, we get
\[ \| \langle s \rangle^{-\alpha} A^{is}x \|_{\gamma(\R,ds,X)}
\cong \left\|(\pi-|\theta|)^{\alpha-\frac12} \frac{1}{\cosh (\pi s)} e^{\theta s}A^{is}x \right\|_{\gamma(\R \times (-\pi,\pi),ds d\theta,X)}.\]
In \cite[p. 228 and Theorem 15.18]{KuWe}, the following formula is derived for $x \in A(D(A^2))$ and $|\theta| < \pi:$
\begin{equation}\label{Equ Res A Mellin BIP}
\frac{\pi}{\sin \pi (\beta + i s)} e^{\theta s} A^{is} x =
\int_0^\infty t^{is}\left[t^{\beta} e^{i\theta \beta} A^{1 - \beta}(e^{i\theta}t+A)^{-1}x\right] \frac{dt}{t}.
\end{equation}
Note that $A(D(A^2))$ is a dense subset of $X.$
As the Mellin transform $f(s) \mapsto \int_0^\infty t^{is} f(s) \frac{ds}{s}$ is an isometry $L^2(\R_+,\frac{ds}s)\to L^2(\R,dt),$ we get by Lemma \ref{Lem square function integral transform} (2)
\begin{align}
\| \langle s \rangle^{-\alpha} A^{is} x\|_{\gamma(\R,X)}
& \cong \|(\pi -|\theta|)^{\alpha-\frac12} t^{\beta}A^{1-\beta}(e^{i\theta}t+A)^{-1} x\|_{\gamma(\R_+\times(-\pi,\pi),\frac{dt}{t}d\theta,X)}
\nonumber \\
& \cong \| |\theta|^{\alpha-\frac12} t^\beta A^{1 - \beta} R(e^{i\theta}t,A) x\|_{\gamma(\R_+ \times (0,2\pi),dt/t d\theta,X)}.
\nonumber
\end{align}

so that (2) $\Leftrightarrow$ (4) for $\theta_0 = \pi.$

For a general $\theta_0 \in (0,\pi],$ consider $K$ from \eqref{Equ Proof Charact K} with restricted image, i.e.
\[K : L^2(\R,ds) \to L^2(\R \times (-\pi,-(\pi-\theta_0)] \cup [\pi-\theta_0,\pi),ds d\theta).\]
Then argue as in the case $\theta_0 = \pi.$\\

\noindent
(4) $\Longleftrightarrow$ (6).\\
The proof of (2) $\Leftrightarrow$ (4) above shows that condition (4) is independent of $\theta_0 \in (0,\pi]$ and $\beta \in (0,1).$
Put $\theta_0 = \pi$ and $\beta = \frac12.$
The equivalence follows again from Lemma \ref{Lem square function integral transform},
using the fact that for $\theta \in (-\frac{\pi}{2},\frac{\pi}{2})$ and $\mu > 0,$
\begin{equation}\label{Equ Proof BIP Sgr Res}
(e^{i\theta} \mu + it)^{-1} = K[\exp(-(\cdot)e^{i\theta} \mu)\chi_{(0,\infty)}(\cdot)](t),
\end{equation}
where $K : L^2(\R,ds) \to L^2(\R,dt)$ is the Fourier transform.\\

\noindent
(3) $\Longleftrightarrow$ (5) for $\beta = \frac12.$\\
We use the same argument as right above.\\

\noindent
(2) $\Longrightarrow$ (3).\\
We use a similar $K_\theta$ as in the proof of $(2) \Leftrightarrow (4),$ fixing $\theta \in (-\pi,\pi):$
\[ K_\theta:  L^2(\R,ds) \to L^2(\R,ds),\, f(s) \mapsto (\pi - |\theta|)^{\alpha} \frac{1}{\sin\pi (\beta + is)}  e^{\theta s} \langle s \rangle^{\alpha} f(s). \]
We have
\[\sup_{|\theta| < \pi} \|K_\theta\| = \sup_{|\theta| < \pi,\,s\in\R} \langle s \rangle^\alpha (\pi-|\theta|)^\alpha
\frac{e^{\theta s}}{|\sin \pi (\beta + is)|} \lesssim \sup_{\theta,s} \langle s(\pi-|\theta|) \rangle^\alpha e^{-|s|(\pi-|\theta|)} < \infty.\]
Thus, by \eqref{Equ Res A Mellin BIP},
\begin{align}
\sup_{0 < |\theta| \leq \pi} |\theta|^{\alpha} \|t^\beta A^{1-\beta}R(te^{i\theta},A)x\|_{\gamma(\R_+,dt/t,X)}
& = \sup_{|\theta| < \pi} (\pi-|\theta|)^{\alpha} \|t^{\beta} A^{1 - \beta} (e^{i\theta} t +A)^{-1}x\|_{\gamma(\R_+,\frac{dt}{t},X)}
\nonumber \\
& = \sup_{|\theta| < \pi} (\pi-|\theta|)^{\alpha} \| \frac{\pi}{\sin\pi (\beta + is)}e^{\theta s}A^{is}x \|_{\gamma(\R,ds,X)}
\label{Equ Proof Thm Charact} \\
& \lesssim \|\langle s \rangle^{-\alpha} A^{is}x\|_{\gamma(\R,ds,X)}.
\nonumber
\end{align}

\noindent
(3), $\alpha$ $\Longrightarrow$ (2), $\alpha + \epsilon.$\\
First we consider $\langle s \rangle^{-(\alpha+\epsilon)}A^{is}x$ for $s\geq 1.$
\begin{align}
\|\langle s \rangle^{-(\alpha+\epsilon)}A^{is}x\|_{\gamma([1,\infty),X)}
& \leq \sum_{n = 0}^\infty 2^{-n\epsilon} \| \langle s \rangle^{-\alpha} A^{is}x\|_{\gamma([2^n,2^{n+1}],X)}.
\label{Equ Proof Thm Charact 2}
\end{align}
For $s \in [2^n,2^{n+1}],$ we have
\[\langle s \rangle^{-\alpha} \lesssim 2^{-n\alpha} \lesssim 2^{-n\alpha} e^{-2^{-n}s} \lesssim (\pi-\theta_n)^\alpha \frac{e^{\theta_n s}}{\sin \pi(\beta + i s)} ,\]
where $\theta_n = \pi - 2^{-n}.$
Therefore
\begin{align}
\| \langle s \rangle^{-\alpha} A^{is}x\|_{\gamma([2^n,2^{n+1}],X)}
& \lesssim (\pi-\theta_n)^{\alpha} \| \frac{\pi}{\sin\pi (\beta + is)}e^{\theta_n s}A^{is}x \|_{\gamma(\R,X)} \nonumber \\
& \overset{\eqref{Equ Proof Thm Charact}}{\lesssim} \sup_{0<|\theta| \leq \pi} |\theta|^{\alpha} \| t^\beta A^{1 - \beta}R(te^{i\theta},A) x\|_{\gamma(\R_+,dt/t,X)} < \infty.
\nonumber
\end{align}
Thus, the sum in \eqref{Equ Proof Thm Charact 2} is finite.

The part $\langle s \rangle^{-(\alpha + \epsilon)}A^{is}x$ for $s \leq -1$ is treated similarly,
whereas $\|\langle s \rangle^{-\alpha}A^{is}x\|_{\gamma((-1,1),X)} \cong \|A^{is}x\|_{\gamma((-1,1),X)}.$
It remains to show that the last expression is finite.
We have assumed that $X$ has property $(\alpha).$
Then the fact that $A$ has an $\HI$ calculus implies that $\{A^{is}:\: |s| < 1\}$ is $\gamma$-bounded \cite[Corollary 6.6]{KaW2}.
Then by Lemma \ref{Lem Folklore square functions} (3), we have $\|A^{is}x\|_{\gamma((-1,1),X)} \leq \gamma\left(\left\{ A^{is} : |s| < 1 \right\}\right) \|1\|_{L^2(-1,1)} \|x\|.$
\end{proof}

Condition (5) of the preceding proposition can be checked in the following way.

\begin{lem}\label{Lem Sgr}
Let $A$ be a $0$-sectorial operator on a space $X$ with property $(\alpha)$ having an $\HI$ calculus.
If for some $\beta > 0$
\begin{equation}\label{Equ Lem Sgr}
\left\{ T \left(te^{\pm i(\frac{\pi}{2} - \theta)} \right) :\: t > 0\right\}\text{ is }\gamma\text{-bounded with constant }\lesssim \theta^{-\beta},
\end{equation}
then $\| A^{\frac12} T(e^{\pm i (\frac{\pi}{2} - \theta)}t) x \|_{\gamma(\R_+,X)} \lesssim \theta^{-\alpha} \|x\|$ with $\alpha = \beta + \frac12.$
\end{lem}

\begin{proof}
Decompose
\[ t \exp \left(\pm i (\frac{\pi}{2} - \theta)\right) = s(t,\theta) + r(t,\theta) \exp\left(\pm i (\frac{\pi}{2} - \frac{\theta}{2})\right) \]
where the reals $s$ and $r$ are uniquely determined by $t$ and $\theta.$
We have $s(t,\theta) = \kappa(\theta) t$ with $\kappa(\theta) \cong \theta.$
Then by Lemma \ref{Lem Folklore square functions},
\begin{align*}
\| A^{\frac12} T(t e^{\pm i (\frac{\pi}{2} - \theta)}) x \|_{\gamma(\R_+,X)} & = \| T(r e^{\pm i (\frac{\pi}{2} - \frac{\theta}{2})}) A^{\frac12} T(s) x\|_{\gamma(\R_+,X)} \\
& \leq \gamma\left(\left\{ T(r e^{\pm i (\frac{\pi}{2} - \frac{\theta}{2})}) :\: r > 0 \right\} \right) \| A^{\frac12} T(s(t,\theta)) x \|_{\gamma(\R_+,X)} \\
& \lesssim (\theta/2)^{-\beta} \theta^{-\frac12} \| A^{\frac12} T(t) x \|_{\gamma(\R_+,X)}.
\end{align*}
By \eqref{Equ Proof BIP Sgr Res} and \cite[Theorem 7.2]{KaW2}, $\| A^{\frac12} T(t) x\|_{\gamma(\R_+,X)} = \| A^{\frac12} (it - A)^{-1} x\|_{\gamma(\R,X)} \leq C \|x\|,$ which finishes the proof.
\end{proof}

Let us now turn to some examples.

\begin{defi}
Let $\Omega$ be a topological space which is equipped with a distance $\rho$ and a Borel measure $\mu.$
Let $d \geq 1$ be an integer.
$\Omega$ is called a homogeneous space\index{homogeneous space} of dimension $d$ if there exists $C > 0$ such that for any $x \in \Omega,\,r> 0$ and $\lambda \geq 1:$
\[\mu(B(x,\lambda r)) \leq C \lambda^d \mu(B(x,r)).\]
\end{defi}

Typical cases of homogeneous spaces are open subsets of $\R^d$ with Lipschitz boundary and Lie groups with polynomial volume growth, in particular stratified nilpotent Lie groups (see e.g. \cite{FoSt}).

We will consider operators satisfying the following assumption.

\begin{assumption}\label{Ass Examples}
$A$ is a self-adjoint positive (injective) operator on $L^2(\Omega),$ where $\Omega$ is a homogeneous space of a certain dimension $d.$
Further, there exists some $p_0 \in [1,2)$ such that the semigroup generated by $-A$ satisfies the so-called generalized Gaussian estimate
(see e.g. \cite[(GGE)]{Bluna}):
\begin{equation}\tag{GGE}\label{Equ generalized Gaussian estimate}
\| \chi_{B(x,r_t)} e^{-tA} \chi_{B(y,r_t)} \|_{p_0 \to p_0'} \leq C \mu(B(x,r_t))^{\frac{1}{p_0'}-\frac{1}{p_0}} \exp \left(-c \left(\rho(x,y)/r_t \right)^\frac{m}{m-1} \right) \quad (x,y \in \Omega,\,t>0).
\end{equation}
Here, $p_0'$ is the conjugated exponent to $p_0,\,C,c > 0,\,m \geq 2$ and $r_t = t^{\frac1m},$ $\chi_B$ denotes the characteristic function of $B,$ $B(x,r)$ is the ball $\{y \in \Omega:\: \rho(y,x) < r\}$
and $\|\chi_{B_1} T \chi_{B_2}\|_{p_0 \to p_0'} = \sup_{\|f\|_{p_0} \leq 1} \|\chi_{B_1} \cdot T(\chi_{B_2} f)\|_{p_0'}.$
\end{assumption}

If $p_0 = 1,$ then it is proved in \cite{BlKub} that \eqref{Equ generalized Gaussian estimate} is equivalent to the usual Gaussian estimate,
i.e. $e^{-tA}$ has an integral kernel $k_t(x,y)$ satisfying the pointwise estimate (cf. e.g. \cite[Assumption 2.2]{DuOS})
\begin{equation}\tag{GE}\label{Equ Gaussian estimate}
|k_t(x,y)| \lesssim \mu(B(x,t^{\frac1m}))^{-1} \exp\left(-c \left(\rho(x,y)/t^{\frac1m}\right)^{\frac{m}{m-1}}\right)\quad (x,y \in \Omega,\,t>0).
\end{equation}
This is satisfied in particular by sublaplacian operators on Lie groups of polynomial growth \cite{Varo} as considered e.g. in
\cite{MaMe,Chri, Alex, MSt,Duon}, or by more general elliptic and sub-elliptic operators
\cite{Davia, Ouha}, and Schr\"odinger operators \cite{Ouhaa}.
It is also satisfied by all the operators in \cite[Section 2]{DuOS}.

Examples of operators satisfying a generalized Gaussian estimate for $p_0 > 1$ are higher order operators with bounded coefficients and Dirichlet boundary conditions on domains of $\R^d,$ Schr\"odinger operators with singular potentials on $\R^d$ and elliptic operators on Riemannian manifolds as listed in \cite[Section 2]{Bluna} and the references therein.

\begin{thm}\label{Thm R-bounded Blunck Hormander thm}
Let Assumption \ref{Ass Examples} hold.
Then for any $p \in (p_0,p_0'),$ the $\HI$ calculus of $A$ extends to a matricially $\gamma$-bounded homomorphism $\Ha \to B(L^p(\Omega))$ with
\[\alpha > d\left| \frac1{p_0} - \frac12\right|+\frac12.\]
\end{thm}

\begin{proof}
We show that \eqref{Equ Lem Sgr} holds with $\beta = d (\frac{1}{p_0} - \frac12).$
By \cite[Proposition 2.1]{BlKua}, the assumption \eqref{Equ generalized Gaussian estimate} implies that
\[ \| \chi_{B(x,r_t)} e^{-tA} \chi_{B(y,r_t)}\|_{p_0\to 2} \leq C_1 \mu(B(x,r_t))^{\frac{1}{2}-\frac{1}{p_0}} \exp(-c_1 (\rho(x,y)/r_t)^{\frac{m}{m-1}}) \quad (x,y \in \Omega,\,t>0)\]
for some $C_1,c_1>0.$
By \cite[Theorem 2.1]{Blun}, this can be extended from real $t$ to complex $z = te^{i\theta}$ with $\theta \in (-\frac{\pi}{2},\frac{\pi}{2}):$
\[ \|\chi_{B(x,r_z)} e^{-zA} \chi_{B(y,r_z)} \|_{p_0 \to 2} \leq C_2 \mu(B(x,r_z))^{\frac{1}{2}-\frac{1}{p_0}} (\cos \theta )^{-d(\frac{1}{p_0} -\frac{1}{2})} \exp(-c_2 (\rho(x,y)/r_z)^{\frac{m}{m-1}}),\]
for $r_z = (\cos \theta )^{-{\frac{m-1}{m}}} t^{\frac1m},$ and some $C_2,c_2 >0.$
By \cite[Proposition 2.1 (i) (1) $\Rightarrow$ (3) with $R = e^{-zA},\,\gamma =\alpha = \frac{1}{p_0} - \frac12,\,\beta = 0,\,r = r_z,\,u=p_0$ and $v = 2$]{BlKua}, this gives for any $x \in \Omega,\,\Re z > 0$ and $k \in \N_0$
\[ \|\chi_{B(x,r_z)} e^{-zA} \chi_{A(x,r_z,k)}\|_{p_0 \to 2} \leq C_3 \mu(B(x,r_z))^{\frac12 - \frac{1}{p_0}} (\cos \theta )^{-d(\frac{1}{p_0} -\frac{1}{2})} \exp(-c_3 k^{\frac{m}{m-1}}),\]
where $A(x,r_z,k)$ denotes the annular set $B(x,(k+1)r_z)\backslash B(x,k r_z).$
By \cite[Theorem 2.2 with $q_0 = p_0,\,q_1 = s = 2,\rho(z) = r_z$ and $S(z) = (\cos \theta)^{d(\frac{1}{p_0}-\frac{1}{2})} e^{-zA}$]{Kuns} and property $(\alpha),$
we deduce that
\[ \{ (\cos \theta)^{d(\frac{1}{p_0}-\frac{1}{2})} e^{-zA} :\: \Re z > 0\} \]
is $\gamma$-bounded. 
Now apply Lemma \ref{Lem Sgr} and Proposition \ref{Prop BIP Sgr Res}, noting that $A$ has an $\HI$ calculus on $L^p(\Omega)$ \cite[Corollary 2.3]{Blun}.
\end{proof}

\begin{rem}\label{Rem R-bounded Blunck Hormander thm}
\begin{enumerate}
\item Theorem \ref{Thm R-bounded Blunck Hormander thm} improves on \cite[Theorem 1.1]{Bluna} in that it includes the matricial $\gamma$-boundedness of the H\"ormander calculus.
Note that \cite{Bluna} obtains also a weak-type result for $p = p_0.$
If $p_0$ is strictly larger than $1,$ then Theorem \ref{Thm R-bounded Blunck Hormander thm} improves the order of derivation $\alpha$ of the calculus from $\displaystyle \frac{d}{2} + \frac{1}{2} +\epsilon$ in \cite{Bluna} to $\displaystyle d \left|\frac{1}{p_0} - \frac{1}{2} \right| + \frac12 + \epsilon.$
In \cite[Theorem 6.4 a)]{Uhl}, under the assumptions of Theorem \ref{Thm R-bounded Blunck Hormander thm}, a $\Hor^\beta_r$ calculus with 
with $\beta > (d+1) | \frac{1}{p_0} - \frac{1}{2} |$ and $r > | \frac12 - \frac1p |^{-1}$ is derived.
Here $\Hor^\beta_r$ is defined similarly to $\Ha$ by
\[ \Hor^\beta_{r} = \{ f : (0,\infty)\to \C:\: \sup_{k \in \Z} \|(f \circ \exp )\phi_k\|_{W^\beta_r} < \infty\}.\]
Note that $\Hor^\beta_r$ is larger than $\Hor^\alpha.$
In the classical case of Gaussian estimates, i.e. $p_0 = 1,$ \cite{DuOS} yields a $\Hor^{\alpha_2}_\infty$ calculus under Assumption \ref{Ass Examples}
and even a $\Hor^{\alpha_2}$ calculus for many examples, e.g. homogeneous operators, with the better derivation order $\alpha_2 > \frac{d}{2}.$
\item The theorem also holds for the weaker assumption that $\Omega$ is an open subset of a homogeneous space $\tilde{\Omega}.$
In that case, the ball $B(x,r_t)$ on the right hand side in \eqref{Equ generalized Gaussian estimate} is the one in $\tilde{\Omega}.$
This variant can be applied to elliptic operators on irregular domains $\Omega \subset \R^d$ as discussed in \cite[Section 2]{Bluna}.
\end{enumerate}
\end{rem}

In Theorem \ref{Thm R-bounded Blunck Hormander thm}, the operator $A$ was assumed to be self-adjoint, and thus, admits a functional calculus $L^\infty \to B(L^2(\Omega)).$
The space $L^\infty = L^\infty((0,\infty);d\mu_A)$ is larger than $\Ha,$ and one can use this fact to ameliorate the functional calculus of $A$ on $L^q(\Omega)$ by complex interpolation.

\begin{prop}
Let $A$ satisfy Assumption \ref{Ass Examples}.
Then for $q \in (p_0,p_0'),\, \alpha > d \left| \frac{1}{p_0}- \frac12 \right| + \frac12$ and $\theta \in (0,1)$ with $\theta > \left| \frac1q -\frac{1}{p_0}\right| / \left|{\frac12 - \frac1{p_0}} \right|,$
the functional calculus of $A$ on $L^q$
\begin{equation}\label{Equ Proof Rem Blunck}
u_{L^q} : (L^\infty,\Ha)_\theta \to B(L^q(\Omega)) \text{ is matricially }\gamma\text{-bounded.}
\end{equation}
Here $(L^\infty,\Ha)_\theta$ is the complex interpolation space which is given an operator space structure \cite[p.~56]{Pis2} by
\[ M_n \otimes (L^\infty,\Ha)_{\theta} = (M_n \otimes L^\infty,M_n \otimes \Ha)_\theta. \]
\end{prop}

\begin{proof}
The self-adjoint calculus $u_{L^2} : L^\infty \to B(L^2(\Omega))$ is completely bounded since it is a $\ast$-representation \cite[Proposition 1.5]{Pis2}, so by Remark \ref{Rem mat}, $u_{L^2}$ is matricially $\gamma$-bounded.
Moreover, we have $(\Gauss(L^p),\Gauss(L^2))_\theta = \Gauss((L^p,L^2)_\theta)$ \cite[Proposition 3.7]{KaKW}.
Then by bilinear interpolation between
\[ M_n \otimes L^\infty \times \Gauss_n(L^2) \to \Gauss_n(L^2),\, ([a_{ij}] \otimes f ,\sum_k \gamma_k \otimes x_k) \mapsto \sum_{k,j} \gamma_k a_{kj} u_{L^2}(f)x_j \]
and, with the mapping $u_{L^p}$ resulting from Theorem \ref{Thm R-bounded Blunck Hormander thm}, $p$ given by $\theta = |\frac1q - \frac1p| / |\frac12 - \frac1p|,$
\[ M_n \otimes \Ha \times \Gauss_n(L^p) \to \Gauss_n(L^p),\, ([a_{ij}] \otimes f , \sum_k \gamma_k \otimes x_k) \mapsto \sum_{k,j} \gamma_k a_{kj} u_{L^p}(f)x_j, \]
one deduces \eqref{Equ Proof Rem Blunck}.
\end{proof}

Note that the space $(L^\infty,\Ha)_{\theta}$ contains $\Hor^\beta_{r,0},$ where $\frac1r > \frac{\theta}{2},$ $\beta > \alpha \theta + (\frac1r - \frac{\theta}{2}).$
Here $\Hor^\beta_{r,0} = \left\{ f \in \Hor^{\beta}_r :\: \|(f \circ \exp) \phi_k \|_{W^\beta_r} \to 0\text{ for }|k| \to \infty \right\}.$
Then \eqref{Equ Proof Rem Blunck} implies that in particular, $\Hor^\beta_{r,0} \to B(L^q),\,f \mapsto f(A)$ is (norm) bounded and by \cite[Section 4.6.1]{Kr}, this extends moreover boundedly to $\Hor^\beta_r \to B(L^q).$

In \cite[Section 7]{DuOS}, for many examples of operators $A$ satisfying \eqref{Equ Gaussian estimate}, it is shown that the functional calculus
\begin{equation}\label{Equ DuOS}
u: \Ha \to B(L^p(\Omega)),\,f \mapsto f(A)\text{ is bounded for }1 < p < \infty\text{ and }\alpha > \frac{d}{2}.
\end{equation}
Moreover, for the fundamental example $A = - \Delta$ on $L^p(\R^d),$ the critical order $\frac{d}{2}$ in \eqref{Equ DuOS} is optimal \cite[IV.7.4]{Ste},\cite[Proposition 4.12 (2)]{Kr}.
Note that the derivation order for the matricially $\gamma$-bounded calculus obtained in Theorem \ref{Thm R-bounded Blunck Hormander thm} under the assumption \eqref{Equ Gaussian estimate} (i.e. $p_0 = 1$)
is only $\frac{d+1}{2},$ and therefore gives a weaker result
in the derivation order compared to \eqref{Equ DuOS}.

Thus the question arises if an arbitrary $A$ that has a norm-bounded H\"ormander calculus also has a matricially $\gamma$-bounded H\"ormander calculus.
In contrast to the self-adjoint $L^\infty$ calculus on Hilbert space, which is always matricially $\gamma$-bounded (see the proof above), we have the following result.

\begin{prop}\label{Prop bounded vs mat-gamma bounded}
Let $A$ be a $0$-sectorial operator on a space $X$ with property $(\alpha).$
Let $\alpha > \frac12$ and $\beta > \alpha + 1.$
Suppose that its functional calculus $f \mapsto f(A)$ is bounded $u_\alpha : \Ha \to B(X),$ and denote $u_\beta$ the restriction of $u_\alpha$ to $\Hor^\beta.$
Then $u_\beta : \Hor^\beta \to B(X)$ is matricially $\gamma$-bounded.

On the other hand, for any $\alpha > 0,$ there exists some $A$ on a Hilbert space $X$ such that
$u_\alpha : \Ha \to B(X)$ is bounded (even $\gamma$-bounded because of \eqref{Equ Hilbert Gaussian}),
but its restriction $u_{\alpha + \frac12} : \Hor^{\alpha + \frac12} \to B(X)$ is not matricially $\gamma$-bounded.
\end{prop}

\begin{proof}
For $t \in \R,$ let $f_t(\lambda) = \lambda^{it}.$
It is easy to check that $\|f_t\|_{\Ha} \lesssim \langle t \rangle^{\alpha}$ \cite[Lemma 4.12 (4)]{Kr}.
By Lemma \ref{Lem Function spaces} (1), $A$ has an $\HI$ calculus.
By \cite[Corollary 6.3]{Kr1}, the set $\left\{ T(te^{\pm i (\frac{\pi}{2} - \theta)}) :\: t > 0 \right\}$ is $\gamma$-bounded with constant $\leq C \theta^{-\alpha - \frac12}\quad (\theta \in (0,\frac{\pi}{2})).$
By Lemma \ref{Lem Sgr}, condition (5) of Proposition \ref{Prop BIP Sgr Res} is satisfied with $\alpha + 1$ in place of $\alpha$ and therefore, $u_{\beta} : \Hor^{\beta} \to B(X)$ is matricially $\gamma$-bounded.

For the second statement, let $\alpha > \frac12.$
Consider $X = \Wa$ and the group $U(t) g = (\cdot)^{it} g$ on $X.$
Note that
\[ \| (\cdot)^{it} g \|_X = \| (g \circ \exp)\hat{\phantom{i}}(\cdot - t) \langle \cdot \rangle^\alpha \|_2 = \| (g \circ \exp)\hat{\phantom{i}} (\cdot) \langle (\cdot) + t \rangle^\alpha \|_2 \cong \langle t \rangle^\alpha \|g\|_X. \]
In particular, $\|U(t)\| \cong \langle t \rangle^\alpha.$
It is easy to check that $U(t) = A^{it}$ are the imaginary powers of a $0$-sectorial operator $A$ and that $f(A) g = f g$ for any $g \in X$ and $f \in \bigcup_{\omega > 0} \HI(\Sigma_\omega).$
By \cite{Str}, one has $\|fg\|_{\Wa} \lesssim \|f\|_{\Ha} \|g\|_{\Wa}.$
Thus, $A$ has a bounded $\Ha$ calculus.

On the other hand, since $X$ is a Hilbert space, the square function condition of Theorem 4.10 reads
\[ \| \langle t \rangle^{-\beta} A^{it} x \|_{\gamma(\R,X)} = \| \langle t \rangle^{-\beta} A^{it} x\|_{L^2(\R,X)} \cong \left( \int_\R \langle t \rangle^{-2 \beta + 2 \alpha} dt \right)^{\frac12} \|x\|, \]
which is finite if and only if $\beta > \alpha + \frac12.$
\end{proof}

\section{Proofs of Lemmas \ref{Lem 1} - \ref{Lem 4}}\label{Sec 6 Proofs Lemmas}

\begin{proof}[Proof of Lemma \ref{Lem 1}]
Since $X$ has property $(\alpha),$ the fact that $A$ has a bounded $\HI(\Sigma_\omega)$ calculus implies \cite[Theorem 12.8]{KuWe} that for any $\theta > \omega,$
\[ \left\{ g(A) :\: \|g\|_{\infty,\theta} \leq 1 \right\}\text{ is }\gamma\text{-bounded}.\]
We fix some $\theta \in (\omega, \frac{\pi}{4}).$
As the mapping $u : \: \Wa \to B(X)$ is matricially $\gamma$-bounded, by Remark \ref{Rem mat},
\[ \left\{ h(A) :\: \|h\|_{\Wa} \leq 1 \right\}\text{ is }\gamma\text{-bounded}.\]
The lemma stated that
$\gamma\left(\left\{ f_{2^k z}(A) :\: k \in \Z \right\}\right) \lesssim \left| z /\Re z\right|^\beta,$ where $f_{2^k z}(\lambda) = \exp(-2^k z \lambda).$
Thus it suffices to decompose $f_{2^k z} = g + h,$ where $\|g\|_{\infty,\theta},\,\|h\|_{\Wa} \lesssim \left| z / \Re z \right|^\beta.$

As $\Psi : f \mapsto f(r \cdot)$ is an isomorphism $\HI(\Sigma_\theta) \to \HI(\Sigma_\theta)$ and $\Wa \to \Wa,$ with $\|\Psi \| \cdot \|\Psi^{-1}\| \leq C,\:C$ independent of $r > 0,$ it suffices to have the above decomposition for $|z| = 1$ and $k = 0.$
We choose $g(\lambda) = \exp(-(z+1) \lambda)$ and $h(\lambda) = \exp(-z\lambda) (1 - e^{-\lambda}).$
As $|\arg(z+1)| + \theta \leq \frac{\pi}{4} + \frac{\pi}{4} = \frac{\pi}{2},$ we actually have $\|g\|_{\infty,\theta} \leq 1 \lesssim \left |\Re z \right|^{-\beta}.$
Further it is a simple matter to check that $\|h\|_{\Wa} \lesssim \left |\Re z \right|^{-\beta}$
for any $\beta > \alpha.$
\end{proof}

\begin{proof}[Proof of Lemma \ref{Lem 2}]
The assumption of the lemma was
\begin{equation}\label{Equ Proof Lem 2}
\gamma\left(\left\{ \exp(-2^k z A) :\: k \in \Z \right\}\right) \lesssim \left| \frac{z}{\Re z} \right|^\beta \quad (\Re z > 0).
\end{equation}
We first show that
\begin{equation}\label{Equ 2 Proof Lem 2}
\gamma \left( \left\{ (2^k t A)^{\frac12} \exp(-2^k t e^{\pm i (\frac{\pi}{2} - \omega) } A ) :\: k \in \Z \right\} \right) \lesssim \omega^{-(\beta + \frac12)} \quad (\omega \in (0,\frac{\pi}{2})).
\end{equation}
Decompose
\[ e^{\pm i (\frac{\pi}{2} - \omega)} t  = s + e^{\pm i (\frac{\pi}{2} - \frac{\omega}{2})} r,\]
where $s,r > 0$ are uniquely determined by $t$ and $\omega.$
Then
\[ (2^k t A)^{\frac12} \exp(-e^{\pm i (\frac{\pi}{2} - \omega)} 2^k t A) = \left( \frac{t}{s} \right)^{\frac12} (2^k s A)^{\frac12} \exp(-2^k s A) \exp(- 2^k r e^{\pm i (\frac{\pi}{2} - \frac{\omega}{2})}A),\]
and consequently,
\begin{align}
\gamma\left(\left\{(2^k t A)^{\frac12} \exp(-e^{\pm i (\frac{\pi}{2} - \omega)}2^k t A) :\: k \in \Z \right\} \right) & \leq \sup_t (t/s)^{\frac12} \times \gamma\left(\left\{ (2^k  s A)^{\frac12} \exp(-2^k s A) :\: k \in \Z \right\} \right) \nonumber \\
& \times \gamma \left(\left\{ \exp(-2^k r e^{\pm i (\frac{\pi}{2} - \frac{\omega}{2})}A):\: k \in \Z\right\} \right).
\label{Equ 3 Proof Lem 2}
\end{align}
We will show that the right hand side of \eqref{Equ 3 Proof Lem 2} can be estimated by $\lesssim \omega^{-\frac12} \times 1 \times \omega^{-\beta}.$
The estimate for the first factor follows from the law of sines
\[ t/s  = \sin (\frac{\pi}{2} + \omega/2)/\sin(\omega/2) \cong \omega^{-1}. \]
For the second estimate, note that by \cite[Example 2.16]{KuWe}, \eqref{Equ Proof Lem 2} implies that $\{ \exp(-zA) :\: z \in \Sigma_\delta \}$ is $\gamma$-bounded for any $\delta < \frac{\pi}{2}$ and consequently, by \cite[Theorem 2.20, $(iii) \Longrightarrow (i)$]{KuWe},
$\{ \lambda (\lambda - A)^{-1} :\: - \lambda \in \Sigma_\theta \}$ is $\gamma$-bounded for any $\theta \in (\frac{\pi}{2},\pi).$
Then with $f(\lambda) = \lambda^{\frac12} e^{-\lambda},$ the Cauchy integral formula \eqref{Equ Cauchy integral formula} gives
\begin{align*}
(2^k t A)^{\frac12} \exp(-2^k t A) & = \frac{1}{2\pi i} \int_{\partial \Sigma_{\pi - \theta}} f(\lambda) (\lambda - 2^ktA)^{-1} d\lambda \\
& = \frac{1}{2\pi i} \int_{\partial \Sigma_{\pi - \theta}} \frac{f(\lambda)}{\lambda} \times \frac{\lambda}{2^k t} (\frac{\lambda}{2^k t} - A)^{-1} d\lambda
\end{align*}
The first factor in the last integral belongs to $L^1(\partial \Sigma_{\pi - \theta},|d\lambda|)$
and the second factor is $\gamma$-bounded by the above for any $\theta < \pi.$
Thus by the well-known integral lemma for $\gamma$-bounds \cite[Corollary 2.14]{KuWe},
the second factor in \eqref{Equ 3 Proof Lem 2} is finite.

The estimate for the third factor in \eqref{Equ 3 Proof Lem 2} follows from the assumption \eqref{Equ Proof Lem 2}, so that we have shown \eqref{Equ 2 Proof Lem 2}.

Now we will write the expression in \eqref{Equ Res} as an integral of the expression in \eqref{Equ 2 Proof Lem 2}.
Let $\theta \in (0,\frac{\pi}{2}),\,\lambda = t e^{i \theta}$ and set $\phi  = \frac{\pi}{2} - \frac{\theta}{2},$ so that $\Re ( e^{i \phi} \lambda ) < 0.$
Then
\begin{align*}
\lambda^{\frac12} (2^k A)^{\frac12} ( \lambda  - 2^k A)^{-1} & = \lambda^{\frac12} (2^k A)^{\frac12} e^{i \phi} (e^{i \phi} \lambda - e^{i \phi} 2^k A)^{-1} \\
& = \int_0^\infty  -e^{i \phi} s^{-\frac12} \lambda^{\frac12} \exp( e^{i \phi} \lambda s) \times (2^k s A)^{\frac12} \exp(-2^k e^{i \phi} s A) ds.
\end{align*}
The second factor of the integral is $\gamma$-bounded by \eqref{Equ 2 Proof Lem 2}
and the first factor is integrable, as the following lines show.
\begin{align*}
\int_0^\infty s^{-\frac12} |\lambda^{\frac12} \exp(e^{i\phi} \lambda s)| ds & = \int_0^\infty s^{-\frac12} | \exp(e^{i \phi} e^{i\theta} s)| ds \\
& = \int_0^\infty s^{-\frac12} \exp( \cos(\frac{\pi}{2} + \frac{\theta}{2}) s) ds \\
& = \int_0^\infty s^{-\frac12} \exp(-s) ds \,|\cos(\frac{\pi}{2} + \frac{\theta}{2})|^{-\frac12} \\
& \lesssim \theta^{-\frac12}
\end{align*}
Then $\tau = \left\{ \lambda^{\frac12}(2^k A)^{\frac12} (\lambda - 2^k A)^{-1} :\: k \in \Z \right\}$ is $\gamma$-bounded since by \cite[Proposition 2.6 (5)]{Kr}, we have
\begin{align*}
\gamma(\tau) & \leq \int_0^\infty s^{-\frac12} | \lambda^{\frac12} \exp(e^{i\phi} \lambda s) | ds \times
\sup_{ t  > 0} \gamma\left( \left\{ (2^k t A)^{\frac12} \exp(-2^k t e^{\pm i (\frac{\pi}{2} - \omega)} A) :\: k \in \Z \right\} \right) \\
& \lesssim |\arg \lambda |^{- \frac12} \times |\arg \lambda |^{-\beta - \frac12}
\end{align*}
The same reasoning applies for $\lambda = t e^{i\theta}$ and $\theta \in (-\frac{\pi}{2},0).$
\end{proof}

\begin{proof}[Proof of Lemma \ref{Lem 3}]
By \cite[Proposition 4.18]{Kr} and \cite[p. 73]{CDMY}, it suffices to show that for some $\delta \in (\gamma,n),$
\begin{equation}\label{Equ 1 Lem Intermediate bad Mihlin calculus}
\|f(A)\| \lesssim \theta^{-\delta} \|f\|_{\infty,\theta}\text{ for any }f \in \bigcup_{\theta > 0} \HI_0(\Sigma_\theta).
\end{equation}
To show \eqref{Equ 1 Lem Intermediate bad Mihlin calculus}, we use the Kalton-Weis characterization of the bounded $\HI(\Sigma_\theta)$ calculus in terms of $\gamma$-bounded operator families (\cite{KaW}, see also \cite[Theorem 12.7]{KuWe}).
More precisely, we follow that characterization in the form of the proof of \cite[Theorem 12.7]{KuWe} and keep track of the dependence of appearing constants on the angle $\theta$.
It is shown there that for $f \in \HI_0(\Sigma_{2\theta}),\,x \in X$ and $x' \in X',$
\begin{align}
|\spr{f(A)x}{x'}| & = |\frac{1}{2\pi i} \int_{\partial\Sigma_\theta} \spr{\lambda^{-\frac12}f(\lambda) A^{\frac12} (\lambda- A)^{-1}x}{x'} d\lambda|
\nonumber \\
& \leq \frac{1}{2\pi} \sum_{j = \pm 1} \int_0^\infty |\spr{f(te^{ij\theta}) (tA)^{\frac12}(e^{ij\theta}- tA)^{-1} x}{x'}|\frac{dt}{t}
\nonumber \\
& = (*).
\nonumber
\end{align}
We put
\[\phi_{j \theta}(\lambda) = \frac{\lambda^{\frac14}(1+\lambda)^{\frac12}}{e^{ij\theta}-\lambda}
\text{ and }\psi(\lambda)= \left(\frac{\lambda}{(1+\lambda)^2}\right)^{\frac18},
\]
so that $(tA)^{\frac12} (e^{ij\theta}-tA)^{-1} = \phi_{j\theta}(tA)\psi(tA)\psi(tA).$
By \cite[Lemma 12.6]{KuWe}, the integral $(*)$ can be controlled by $\Gauss$-norms.
More precisely, we have
\begin{align}
(*) & \lesssim \sup_{j=\pm 1} \sup_{t > 0} \sup_N  \| \sum_{k=-N}^N \gamma_k \otimes f(2^k t e^{ij \theta}) \phi_{j \theta}(2^k t A) \psi(2^k t A) x\|_{\Gauss(X)}
\label{Equ 2 Lem Intermediate bad Mihlin calculus} \\
& \cdot \|\sum_{k=-N}^N \gamma_k \otimes \psi(2^k t A)'x'\|_{\Gauss(X')}
\nonumber \\
& \lesssim \|f\|_{\infty,\theta}  \sup_{j,t} \gamma\left(\{ \phi_{j \theta}(2^k t A) :\: k \in \Z \}\right) \sup_{N,t} \| \sum_{k=-N}^N \gamma_k \otimes \psi(2^k t A)x\|_{\Gauss(X)}
\nonumber \\
& \cdot \sup_{N,t}\| \sum_{k = -N}^N \gamma_k \otimes \psi(2^k t A)'x'\|_{\Gauss(X')}.
\nonumber
\end{align}
By \cite[Theorem 12.2]{KuWe}, the fact that $A$ has a bounded $\HI$ calculus implies that
$\sup_{N,t} \| \sum_{k=-N}^N \gamma_k \otimes \psi(2^k t A) x\|_{\Gauss(X)} \lesssim \|x\|$ and
$\sup_{N,t} \| \sum_{k=-N}^N \gamma_k \otimes \psi(2^k t A)' x'\|_{\Gauss(X')} \lesssim \|x'\|.$
Note that there is no dependence on $\theta$ in these two inequalities.
It remains to show that
\begin{equation}\label{Equ R-bound of phi_theta}
 \sup_{j = \pm 1,t> 0}\gamma(\{ \phi_{j \theta}(2^k t A) :\: k \in \Z \}) \lesssim \theta^{-\delta}.
\end{equation}
We have
\begin{align}
 \phi_{j \theta}(2^ktA) & = \frac{1}{2\pi i} \int_{\partial\Sigma_{\frac{\theta}{2}}} \phi_{j\theta}(\lambda) \lambda^{\frac12} (2^ktA)^{\frac12} (\lambda-2^ktA)^{-1} \frac{d\lambda}{\lambda}
\nonumber \\
& = \frac{1}{2\pi i} \int_{\partial\Sigma_{\frac{\theta}{2}}} \phi_{j\theta}(t\lambda) \lambda^{\frac12} (2^kA)^{\frac12} (\lambda-2^kA)^{-1} \frac{d\lambda}{\lambda}.
\nonumber
\end{align}
By \cite[Proposition 2.6 (5)]{Kr},
\begin{align*}
\sup_{j= \pm 1,t>0}\gamma(\{\phi_{j\theta}(2^ktA):\:k \in \Z\}) & \lesssim \sup_{j = \pm 1,t > 0} \|\phi_{j \theta}(t\lambda)\|_{L^1(\partial\Sigma_{\frac{\theta}{2}},|\frac{d\lambda}{\lambda}|)} \\
& \times
\sup_{\lambda \in \partial\Sigma_{\theta/2} \backslash \{ 0 \} }\gamma\left(\{ \lambda^{\frac12} (2^kA)^{\frac12} (\lambda - 2^kA)^{-1} :\:  k \in \Z \}\right).
\end{align*}
By assumption, it suffices to show that for any $\epsilon > 0$
\[\sup_{t > 0} \|\phi_{j \theta}(t\lambda)\|_{L^1(\partial\Sigma_{\frac{\theta}{2}},|\frac{d\lambda}{\lambda}|)} \leq C_\epsilon \theta^{-\epsilon}.\]
\begin{align}
\int_{\partial\Sigma_{\frac{\theta}{2}}} |\phi_{j \theta}(t\lambda)| \left|\frac{d\lambda}{\lambda}\right|
& = \int_{\partial\Sigma_{\frac{\theta}{2}}} | \phi_{j \theta}(\lambda)| \left| \frac{d\lambda}{\lambda}\right|
= \sum_{l = \pm 1}\int_0^\infty \left|\frac{s^\frac14 (1 + e^{i l \frac{\theta}{2}} s)^{\frac12}}{e^{i j\theta} - e^{il\frac{\theta}{2}} s}\right| \frac{ds}{s}.
\nonumber
\end{align}
The denominator is estimated from below by
\begin{align}
|e^{ij \theta} - e^{il \frac{\theta}{2}} s | & = |e^{i\theta (j - \frac{l}{2})} - s|
 \gtrsim |\cos(\theta (j- \frac{l}{2})) - s| + |\sin(\theta (j-\frac{l}{2}))|
\nonumber \\
& \gtrsim |1-s| - |\cos(\theta (j-\frac{l}{2}))-1| + \theta
\nonumber \\
& \gtrsim |1-s| - \theta^2 + \theta
\gtrsim |1-s| + \theta
\nonumber
\end{align}
for the crucial case of small $\theta.$
Thus
\[ \int_{\partial\Sigma_{\frac{\theta}{2}}} |\phi_{j \theta}(\lambda)| \left|\frac{d\lambda}{\lambda}\right|
 \lesssim \int_0^\infty \frac{s^{\frac14} (1+s)^{\frac12}}{\theta + |1-s|} \, \frac{ds}{s}.
\]
We split the integral into the parts $\int_0^\infty = \int_0^{\frac12} + \int_{\frac12}^{1-\theta} + \int_{1-\theta}^{1+\theta} + \int_{1+\theta}^2 + \int_2^\infty.$
\[ \int_0^{\frac12} \frac{s^{\frac14}(1+s)^{\frac12}}{\theta + |1-s|} \, \frac{ds}{s} \leq \int_0^{\frac12} \frac{s^{\frac14}(1+s)^{\frac12}}{|1-s|} \, \frac{ds}{s} < \infty\]
is independent of $\theta.$
The same estimate applies to $\int_2^\infty.$
\[ \int_{\frac12}^{1-\theta} \frac{s^{\frac14}(1+s)^{\frac12}}{\theta + |1 - s|} \, \frac{ds}{s} \lesssim \int_{\frac12}^{1-\theta} \frac{1}{\theta + |1 -s|} ds \leq \int_{\frac12}^{1-\theta}  \frac{1}{1-s} ds \lesssim | \log \theta |.\]
Similarly,
\[ \int_{1 + \theta}^2 \frac{s^{\frac14}(1+s)^{\frac12}}{\theta + |1 - s|} \, \frac{ds}{s} \lesssim \int_{1 + \theta}^2 \frac{1}{s-1} ds \lesssim |\log \theta|.\]
Finally,
\[ \int_{1-\theta}^{1 + \theta} \frac{s^{\frac14}(1+s)^{\frac12}}{\theta + |1 - s|} \, \frac{ds}{s} \lesssim \int_{1-\theta}^{1 + \theta} \frac{1}{\theta} ds \lesssim 1.\]
Since $1 + |\log \theta| \leq C_\epsilon \theta^{-\epsilon},$ the lemma is shown.
\end{proof}

\begin{proof}[Proof of Lemma \ref{Lem 4}]
Denote $N = \sup_{x > 0} \# \{ k \in \Z: \: \supp g_k \cap [\frac12 x, 2x] \neq \emptyset \} < \infty.$
Fix $x > 0$ and $j \in \{ 0, 1 ,\ldots, n\}.$
Then almost all $g_k$ vanish in a neighborhood of $x,$ and thus
\[ \left|x^j \frac{d^j}{dx^j} \left(\sum_{k \in \Z} g_k\right)(x)\right| = \left|\sum_{k \in \Z} x^j \frac{d^j}{dx^j}g_k(x)\right| \leq N \sup_{k \in \Z} |x^j \frac{d^j}{dx^j}g_k(x)| \leq N \sup_{k \in \Z} \|g_k\|_{\M^n} .\]
Taking the supremum over $x$ and $j$ gives $\|\sum_{k \in \Z} g_k\|_{\M^n} \leq N \sup_{k \in \Z} \|g_k\|_{\M^n}.$
\end{proof}

\section*{Acknowledgment}
I would like to thank Lutz Weis for several discussions on the subject of this article.
Further, I acknowledge financial support from the Karlsruhe House of Young Scientists KHYS and the Franco-German University DFH-UFA.


\begin{thebibliography}{100}

\bibitem{Alex}
G.~Alexopoulos.
\newblock Spectral multipliers on Lie groups of polynomial growth.
\newblock {\em Proc. Am. Math. Soc.} 120(3):973--979, 1994.

\bibitem{BeL}
J.~Bergh and J.~L\"ofstr\"om.
\newblock {\em Interpolation spaces. An introduction.}
\newblock Grundlehren der mathematischen Wissenschaften, 223.
  Berlin etc.: Springer, 1976.

\bibitem{Bluna}
S.~Blunck.
\newblock A H\"ormander-type spectral multiplier theorem for operators without
  heat kernel.
\newblock {\em Ann. Sc. Norm. Sup. Pisa (5)} 2(3):449--459, 2003.

\bibitem{Blun}
S.~Blunck.
\newblock Generalized Gaussian estimates and Riesz means of Schr\"odinger
  groups.
\newblock {\em J. Aust. Math. Soc.} 82(2):149--162, 2007.

\bibitem{BlKua}
S.~Blunck and P.~Kunstmann.
\newblock Generalized Gaussian estimates and the Legendre transform.
\newblock {\em J. Oper. Theory} 53(2):351--365, 2005.

\bibitem{BlKub}
S.~Blunck and P.~C. Kunstmann.
\newblock Calder\'on-Zygmund theory for non-integral operators and the
  $H^\infty$ functional calculus.
\newblock {\em Rev. Mat. Iberoam.} 19(3):919--942, 2003.

\bibitem{Chri}
M.~Christ.
\newblock $L\sp p$ bounds for spectral multipliers on nilpotent groups.
\newblock {\em Trans. Am. Math. Soc.} 328(1):73--81, 1991.

\bibitem{CDMY}
M.~Cowling, I.~Doust, A.~McIntosh and A.~Yagi.
\newblock Banach space operators with a bounded $H^\infty$ functional
  calculus.
\newblock {\em J. Aust. Math. Soc., Ser. A} 60(1):51--89, 1996.

\bibitem{Davia}
E.~Davies.
\newblock {\em Heat kernels and spectral theory.}
\newblock Cambridge Tracts in Mathematics, 92. Cambridge etc.: Cambridge
  University Press, 1989.

\bibitem{DiJT}
J.~Diestel, H.~Jarchow and A.~Tonge.
\newblock {\em Absolutely summing operators.}
\newblock Cambridge Studies in Advanced Mathematics, 43. Cambridge: Cambridge
  Univ. Press, 1995.

\bibitem{dPR}
B.~de Pagter and W.~Ricker.
\newblock $C(K)$-representations and $R$-boundedness.
\newblock{\em J. London Math. Soc.} (2) 76:498--512, 2007.

\bibitem{Duon}
X.~T. Duong.
\newblock From the $L\sp 1$ norms of the complex heat kernels to a H\"ormander
  multiplier theorem for sub-Laplacians on nilpotent Lie groups.
\newblock {\em Pac. J. Math.} 173(2):413--424, 1996.

\bibitem{DR}
X.~Duong and D.~Robinson.
\newblock Semigroup kernels, Poisson bounds, and holomorphic functional calculus.
\newblock {\em J. Funct. Anal.} 142 No.~1:89--128, 1996.

\bibitem{DuOS}
X.~T. Duong, E.~M. Ouhabaz and A.~Sikora.
\newblock Plancherel-type estimates and sharp spectral multipliers.
\newblock {\em J. Funct. Anal.} 196(2):443--485, 2002.

\bibitem{ER}
E.~Effros and Zh.-J.~Ruan.
\newblock {\em Operator spaces.}
\newblock London Mathematical Society Monographs, New Series, 23. The Clarendon Press, Oxford University Press, New York, 2000.

\bibitem{FoSt}
G.~Folland and E.~Stein.
\newblock {\em Hardy spaces on homogeneous groups.}
\newblock Mathematical Notes, 28. Princeton, NJ: Princeton University
  Press, University of Tokyo Press, 1982.

\bibitem{Frohl}
A.~Fr{\"o}hlich.
\newblock {\em $H^\infty$-Kalk\"ul und Dilatationen}.
\newblock PhD thesis, Universit\"at Karlsruhe, 2003.

\bibitem{HaKu}
B.~Haak and P.~Kunstmann.
\newblock Admissibility of unbounded operators and wellposedness of linear systems in Banach spaces.
\newblock{\em Integral equations Oper. Theory} 55 No.~4:497--533, 2006.

\bibitem{Haa}
M.~Haase.
\newblock {\em The functional calculus for sectorial operators.}
\newblock Operator Theory: Advances and Applications, 169. Basel: Birkh\"auser, 2006.

\bibitem{Hor}
L.~H\"ormander.
\newblock Estimates for translation invariant operators in $L^p$ spaces.
\newblock{\em Acta Math.} 104:93--140, 1960.

\bibitem{KaKW}
N.~Kalton, P.~Kunstmann and L.~Weis.
\newblock Perturbation and interpolation theorems for the $H^\infty$-calculus with applications to differential operators.
\newblock{\em Math. Ann.} 336, no. 4:747--801, 2006.

\bibitem{KNVW}
N.~Kalton, J.~van Neerven, M.~Veraar and L.~Weis.
\newblock Embedding vector-valued Besov spaces into spaces of
  $\gamma$-radonifying operators.
\newblock {\em Math. Nachr.} 281(2):238--252, 2008.

\bibitem{KaW2}
N.~Kalton and L.~Weis.
\newblock The $H^\infty$-calculus and square function estimates, preprint.

\bibitem{KaW}
N.~Kalton and L.~Weis.
\newblock The $H^\infty$-calculus and sums of closed operators.
\newblock {\em Math. Ann.} 321(2):319--345, 2001.

\bibitem{Kr}
C.~Kriegler.
\newblock Spectral multipliers, $R$-bounded homomorphisms, and analytic diffusion semigroups.
\newblock {PhD-thesis, online at
http://digbib.ubka.uni-karlsruhe.de/volltexte/1000015866}

\bibitem{Kr1}
C.~Kriegler.
\newblock Functional calculus and dilation for $c_0$-groups of polynomial growth, submitted.

\bibitem{KrLM}
C.~Kriegler and C.~Le Merdy.
\newblock Tensor extension properties of $C(K)$-representations and applications to unconditionality.
\newblock{\em J.~Aust.~Math.~Soc.} 88:205--230, 2010.

\bibitem{Kuns}
P.~C. Kunstmann.
\newblock On maximal regularity of type $L^p-L^q$ under minimal assumptions for
  elliptic non-divergence operators.
\newblock {\em J. Funct. Anal.} 255(10):2732--2759, 2008.

\bibitem{KuWe}
P.~C. Kunstmann and L.~Weis.
\newblock Maximal $L_p$-regularity for parabolic equations, Fourier multiplier
  theorems and $H^\infty$-functional calculus.
\newblock {\em Functional analytic methods for
  evolution equations. Based on lectures given at the autumn school on
  evolution equations and semigroups, Levico Terme, Trento, Italy, October
  28--November 2, 2001.} Berlin: Springer, Lect. Notes Math. 1855,
  65--311, 2004.

\bibitem{LM04}
C.~Le~Merdy.
\newblock On square functions associated to sectorial operators.
\newblock{\em Bull.~Math.~Soc.~France} 132 (1):137--156, 2004.

\bibitem{LM10}
C.~Le~Merdy.
\newblock $\gamma$-bounded representations of amenable groups.
\newblock{\em Adv.~Math.} 224 No.~4:1641--1671, 2010.

\bibitem{MaMe}
G.~Mauceri and S.~Meda.
\newblock Vector-valued multipliers on stratified groups.
\newblock {\em Rev. Mat. Iberoam.} 6(3-4):141--154, 1990.

\bibitem{MSt}
D.~M\"uller and E.~Stein.
\newblock On spectral multipliers for Heisenberg and related groups.
\newblock {\em J. Math. Pures Appl. IX.} 73(4):413--440, 1994.

\bibitem{Ouha}
E.~M. Ouhabaz.
\newblock {\em Analysis of heat equations on domains.}
\newblock London Mathematical Society Monographs, 31. Princeton, NJ:
  Princeton University Press, 2005.

\bibitem{Ouhaa}
E.~M. Ouhabaz.
\newblock Sharp Gaussian bounds and $L^p$-growth of semigroups associated with
  elliptic and Schr\"odinger operators.
\newblock {\em Proc. Am. Math. Soc.} 134(12):3567--3575, 2006.

\bibitem{Paz}
A.~Pazy.
\newblock Semigroups of linear operators and applications to partial differential equations.
\newblock{\em Applied Mathematical Sciences} 44, New York: Springer Verlag, 1983.

\bibitem{Pis}
G.~Pisier.
\newblock Some results on Banach spaces without local unconditional structure.
\newblock {\em Compos. Math.} 37:3--19, 1978.

\bibitem{Pis2}
G.~Pisier.
\newblock {\em Introduction to Operator Space Theory.}
\newblock London Math. Soc. Lecture Notes Series, 294:
Cambridge University Press, Cambridge, 2003.

\bibitem{Ste1}
E.~Stein.
\newblock {\em Topics in Harmonic Analysis Related to Littlewood-Paley theory.}
\newblock Ann. Math. Stud. 63, Princeton Univ. Press, 1970.

\bibitem{Ste}
E.~Stein.
\newblock {\em Singular integrals and differentiability properties of functions.}
\newblock Princeton Mathematical Series, No. 30 Princeton University Press, 1970.

\bibitem{Str}
R.~Strichartz.
\newblock Multipliers on fractional Sobolev spaces.
\newblock {\em J. Math. and Mech.} 16, 1031--1060, 1967.

\bibitem{vN}
J.~van Neerven.
\newblock $\gamma$-radonifying operators: a survey.
\newblock {\em Proc. Centre Math. Appl. Austral. Nat. Univ.} 44:1--61, 2010.

\bibitem{Uhl}
M.~Uhl.
\newblock Spectral multiplier theorems of Hörmander type via generalized Gaussian estimates.
\newblock {PhD-thesis, online at
http://digbib.ubka.uni-karlsruhe.de/volltexte/1000025107}

\bibitem{Varo}
N.~Varopoulos.
\newblock Analysis on Lie groups.
\newblock {\em J. Funct. Anal.} 76(2):346--410, 1988.

\end{thebibliography}
\end{document}